\documentclass[12pt]{amsart}
\usepackage{amsfonts}
\usepackage{amssymb,amsthm,amscd,amstext,amsmath,mathrsfs,latexsym}
\usepackage{tikz}
\usetikzlibrary{intersections, decorations.pathmorphing, decorations.fractals, decorations.markings}
\newtheorem{theorem}{Theorem}
\newtheorem{lemma}{Lemma}

\newtheorem{prop}{Proposition}

\newtheorem*{mainthm}{Main Theorem}

\usepackage{hyperref}
\hypersetup{
  colorlinks = true,
  allcolors = blue,
}

\usepackage[T1]{fontenc}
\usepackage{mlmodern} 

\usepackage{caption}
\makeatletter
\renewcommand{\@seccntformat}[1]{\csname the#1\endcsname.\ }  % add dot visually
\makeatother

\makeatletter
\def\numberline#1{\hb@xt@\@tempdima{#1.\hfil}}
\makeatother

\begin{document}

\title[Gauge invariance of Kuperberg invariant]
{On the gauge invariance of the Kuperberg invariant of certain high genus framed 3-manifolds}

\author{Liang Chang}
\address{School of Mathematical Sciences and LPMC\\
	Nankai University \\
	Tianjin, China}
\email{changliang996@nankai.edu.cn}

\author{Yilong Wang}
\address{Beijing Institute of Mathematical Sciences and Applications (BIMSA)\\
	Beijing, China}
\email{wyl@bimsa.cn}

\author{Saifei Zhai}
\address{School of Mathematical Sciences and LPMC\\
	Nankai University \\
	Tianjin, China}
\email{saifeizhai@mail.nankai.edu.cn}

\begin{abstract}
We show that the Kuperberg invariant of the Weeks manifold with any framing is a gauge invariant of finite-dimensional Hopf algebras, which provides the first example of gauge invariants of general finite-dimensional Hopf algebras via hyperbolic 3-manifolds. We also show that the Kuperberg invariant of the 3-torus is gauge invariant, which further supports the idea of systematically producing gauge invariants of Hopf algebras via topological methods proposed in \cite{CNW25}.
\end{abstract}

\maketitle

\section{Introduction}
Gauge invariants of finite-dimensional Hopf algebras are assignments of finite-dimensional Hopf algebras to numbers that are invariant under tensor equivalences of the representation categories of Hopf algebras. Gauge invariants are also called categorical invariants. Many natural invariants of finite-dimensional Hopf algebras, including the dimension, the exponent \cite{EG99}, the quasi-exponent \cite{EG02qexp} (for Hopf algebras over $\mathbb{C}$), and the generalization of Frobenius-Schur indicators studied by Kashina, Montgomery and Ng \cite{KMN12} and Shimizu \cite{Shi15}, are gauge invariants. The importance of these gauge invariants in the study of finite-dimensional Hopf algebras, especially in their classification, naturally demands us to look for new constructions of gauge invariants.

Inspired by the intimate relationship among the Frobenius-Schur indicators of a semisimple Hopf algebra $H$ over $\mathbb{C}$, the Turaev-Viro invariants of lens spaces associated to the category $\operatorname{Rep}(H)$ \cite{TV92} and the (involutory) Kuperberg invariant of lens spaces associated to $H$ \cite{Kup91,BW95}, the authors of \cite{CNW25} were led to the study of the gauge invariance of Kuperberg invariants of framed 3-manifolds associated to finite-dimensional Hopf algebras. It is shown in op.~cit.~that the non-semisimple indicators in \cite{KMN12} can be realized as the Kuperberg invariants of certain framed lens spaces. Moreover, the Kuperberg invariant of any framed lens space is a gauge invariant of finite-dimensional Hopf algebras over arbitrary base fields. This result motivated Question 1.1 in \cite{CNW25}: Are all framed Kuperberg invariants gauge invariants? An affirmative answer to this question would not only imply that we can systematically construct infinite families of gauge invariants of finite-dimensional Hopf algebras, but also indicates a potential categorical construction of such framed 3-manifold invariants, which is desirable in the study of non-semisimple quantum invariants and topological field theories \cite{Lyu95I,CGPT20-kup, DGGPR22, MSWY23}. In addition to the lens spaces, the Kuperberg invariants of a family of framed 3-manifold of genus 2 are also shown to be gauge invariants in \cite{CNW25}. One soon realizes that the 3-manifolds studied in \cite{CNW25} are non-hyperbolic, so it is natural to ask: Is the gauge invariance of the framed Kuperberg invariants related to the geometry of 3-manifolds? Moreover, in view of the fact that most 3-manifolds are hyperbolic \cite{Thur82}, it is hard to disagree that a decisive supporting evidence for an affirmative answer to \cite[Question 1.1]{CNW25} for all 3-manifolds would come from the gauge invariance of the Kuperberg invariant of  framed hyperbolic 3-manifolds.

In this paper, we prove the gauge invariance of the Kuperberg invariant of the Weeks manifold (with framings), which is the unique hyperbolic 3-manifold of smallest volume \cite{GMM09}. More precisely, our main theorem is 

\begin{mainthm}(Theorem \ref{weeksthm})
Let $H$ a finite-dimensional Hopf algebra over an arbitrary base field. For any framing $f$ on the Weeks manifold $W$, the Kuperberg invariant $Z(W, f, H)$ is a gauge invariant of $H$.
\end{mainthm}

This is the first example of gauge invariants of general (not necessarily semisimple) finite-dimensional Hopf algebras via hyperbolic 3-manifolds. To prove the main theorem, we first design a Heegaard diagram of the Weeks manifold with two vector fields and show that it is a framed Heegaard diagram (i.e., it is given by a framing $f_0$ on the Weeks manifold). Then we prove the gauge invariance of the Kuperberg invariant $Z(W, f_0, H)$ of the framed Heegaard diagram by showing that it is invariant under gauge transformations. Finally, we argue that the Kuperberg invariant of all other framings on the Weeks manifold are all different from the gauge invariant $Z(W, f_0, H)$ by a multiple of another gauge invariant, hence complete the proof. As is remarked above, our result in this paper provides a very important evidence for the gauge invariance of framed Kuperberg invariants of all 3-manifolds. Finally, in the last section, we check the gauge invariance of the Kuperberg invariant of the 3-torus with a framing, which is a 3-manifold of genus 3.

The paper is organized as follows. In Sections 2, we review basic concepts related finite-dimensional Hopf algebras. In Section 3, we review the definition of the Kuperberg invariant of framed 3-manifolds. In Section 4, we prove our main theorem on the gauge invariance of the Kuperberg invariants of the Weeks manifold. Finally, in Section 5, we prove the gauge invariance of the Kuperberg invariant of the 3-torus with a framing.

\section*{Acknowledgements}
We thank Siu-Hung Ng for fruitful discussions. Y.~W.~is supported by NSFC Grant 12301045 and 12571041, and Beijing Natural Science Foundation Key Program Grant Z220002.

\section{Hopf algebra and gauge invariant}
In this section, we recall the basis structure of Hopf algebras and their gauge transformations. All algebras in this paper are assumed to be finite-dimensional over its base field. 

Let $H$ be a finite-dimensional Hopf algebra over a field $k$ with multiplication $m: H\otimes H\rightarrow H$, unit $1_H:k\rightarrow H$, comultiplication $\Delta: H\rightarrow H\otimes H$, counit $\varepsilon: k\rightarrow H$ and antipode
$S: H\rightarrow H$. We will use the Sweedler notation $\Delta(x)=\sum_{(x)}x_{(1)}\otimes x_{(2)}$ or 
$\Delta(x)=x_{(1)}\otimes x_{(2)}$ with the summation suppressed. 
Iterated comultiplication is defined inductively by
\[\begin{aligned}
	\Delta^0=\varepsilon\circ 1_H\,,\quad \Delta^1=\mathrm{id}_H\,,\quad \Delta^2=\Delta\,, \quad \Delta^{n+1}=(\mathrm{id}\otimes\Delta^{(n)})\circ \Delta \quad\text{for $n \ge 2$.}
\end{aligned}\]
The suppressed Sweedler notation is $\Delta^n(x)=x_{(1)}\otimes\cdots\otimes x_{(n)}$. An element $x\in H$ is called grouplike if $\Delta(x)=x\otimes x$.

A left (resp.~right) integral in $H$ is an element $\Lambda^L$ (resp.~$\Lambda^R$) in $H$ such that
\[\begin{aligned}
	x\Lambda^L=\varepsilon(x)\Lambda^L \quad (\text{resp.~}\Lambda^Rx=\varepsilon(x)\Lambda^R)\,.
\end{aligned}\] 
It turns out that the space of left and right integrals are both 1-dimensional \cite{Sweedler1969}. The dual space $H^*$ of $H$ is also a Hopf algebra with structure maps dual to those of $H$. A left integral $\lambda^L$ (resp.~right integral $\lambda^R$) in $H^*$ is also called a left cointegral (resp.~right cointegral) of $H$, which satisfies 
\[\begin{aligned}
	x_{(1)}\lambda^L(x_{(2)})=\lambda^L(x)1_H,~~~~~~~~\lambda^R(x_{(1)})x_{(2)}=\lambda^R(x)1_H
\end{aligned}\]
for any $x\in H$. Since the integrals exist up to scalar, there are distinguished grouplike elements $g\in H$ and $\alpha\in H^*$ such that
\[\begin{aligned}
	\Lambda^Lx=\alpha(x)\Lambda^L\,,\quad\text{and}\quad\lambda^R(x_{(1)})x_{(2)}=\lambda^R(x)g
\end{aligned}\]
for all $x \in H$.

It is well-known that $\lambda^R(\Lambda^L)\neq 0$ \cite{Rad90}. In this paper, we fix a choice a pair of normalized integrals, namely, a pair of integrals $\Lambda^L \in H$ and $\lambda^R \in H^*$ such that $\lambda^R(\Lambda^L)=1$, and we write $\Lambda=\Lambda^L$ and $\lambda=\lambda^R$ for short. Starting with the normalized integrals, we will have associated integrals
\[\begin{aligned}
	\Lambda^R:=S(\Lambda),~~~~~~~~\lambda^L:=\lambda\circ S=g\rightharpoonup\lambda
\end{aligned}\]
such that $\Lambda^R$ is a right integral of $H$ and $\lambda^L$ is a left integral of $H^*$, and $\lambda(\Lambda^R)=\lambda^L(\Lambda^R)=1$, $\lambda^L(\Lambda^L)=\alpha(g)$. We will use the notation $g\rightharpoonup\lambda\in H^*$, $g\rightharpoonup\lambda(x)=\lambda(xg)$ for any $x \in H$. The following theorem on the trace function plays an important role in our proof for the gauge invariance of Kuperberg invariants. It proof can be found in \cite{Rad94} and \cite{CNW25}.  

\begin{theorem}\label{integral}
Let $H$ be a finite-dimensional Hopf algebra, $\Lambda$ and $\lambda$ is a pair of normalized integrals such that $\lambda(\Lambda)=1$, $X\in \mathrm{End}(H)$ \\
\noindent$(1)$ $S(a)\Lambda_{(1)}\otimes \Lambda_{(2)}=\Lambda_{(1)}\otimes a\Lambda_{(2)}$.\\
\noindent$(2)$ $\mathrm{Tr}(X)=\lambda(S\circ X(\Lambda_{(2)})\Lambda_{(1)})=\lambda(S(\Lambda_{(2)})X(\Lambda_{(1)})).$\\
\noindent$(3)$ $\lambda S(aX(\Lambda_{(1)})\Lambda_{(2)})=\lambda S(X(\Lambda_{(1)}S(a))\Lambda_{(2)})$.\qed
\end{theorem}

The defining feature of Hopf algebras is that their representation categories are finite tensor categories. Two Hopf algebras $H$ and $K$ are said to be gauge equivalent if their representation categories are equivalent as monoidal categories. A quantity $Z(H)$ is called a gauge invariant if $Z(H)=Z(K)$ whenever the Hopf algebra $H$ and $K$ are gauge equivalent. 

Alternatively, one has the following intrinsic characterization of gauge equivalence. An invertible element $F=\sum\limits_{i}f^{[1]}_i\otimes f^{[2]}_i\in H\otimes H$ is called a normalized 2-cocycle if it satisfies
\[
\begin{aligned}
	(F\otimes 1)(\Delta\otimes \mathrm{id})(F)=(1\otimes F)(\mathrm{id}\otimes\Delta)(F)\,,\quad (\varepsilon\otimes\mathrm{id})(F)=(\mathrm{id}\otimes\varepsilon)(F)=1.
\end{aligned}
\]
With a 2-cocycle $F$, we can define a new Hopf algebra structure on $H$ \cite{KMN12}. Namely, $H_F$, called (Drinfeld) twist of $H$, is a Hopf algebra with the same algebra structure as $H$. Its comultiplication and antipode are given as follows:
\[
\begin{aligned}
\Delta_F(x)=F\Delta(x)F^{-1},~~~~~~\varepsilon_F(x)=\varepsilon(x),~~~~~S_F(x)=uS(x)u^{-1}
\end{aligned}
\]
where $u:=\sum\limits_{i}f^{[1]}_iS(f^{[2]}_i)$. Let $F^{-1}=\sum\limits_{j}d^{[1]}_j\otimes d^{[2]}_j$. Then $u^{-1}=\sum\limits_{j}S(d^{[1]}_j)d^{[2]}_j$. In addition, $S_F^2(x)=QS^2(x)Q^{-1}$ where $Q:=uS(u^{-1})$.

The iterated comultiplication $\Delta^n_F$ can be computed in terms of the $n$-fold tensor $F_n$ that is defined recursively: $F_1=1$, $F_2=F$, $F^{-1}_2=F^{-1}$ and
\[\begin{aligned}
F_{n+1}=(1\otimes F_n)(\mathrm{id}\otimes\Delta^n)(F),~~~~F^{-1}_{n+1}=(1\otimes F^{-1}_n)(\mathrm{id}\otimes\Delta^n)(F^{-1}).
\end{aligned}\]
Then $\Delta_F^n(x)=F_n\Delta(x)F^{-1}_n$ \cite{KMN12}.

It is shown in \cite{NS08} (see also \cite{Sch96,EG02tri}) that two Hopf algebras $H$ and $K$ are gauge equivalent if and only if there exists a normalized 2-cocycle $F$ such that $H_F$ is isomorphic to $K$.  In order to find gauge invariant for $H$, we need to construct $Z(H)$ that keeps unchanged under gauge transformation. That is $Z(H)=Z(H_F)$ for any 2-cocycle $F$. 

In the following proposition, we collect some properties of 2-cocycles proven in \cite{CNW25}, which will be used in the rest of the paper (as is noted before, the summation notion will be suppressed for simplicity).

\begin{prop}[{\cite[Lem.\ 5.2]{CNW25}}]\label{cocycle}
Let $F_m=f^{[1]}_i\otimes\cdots\otimes f^{[m]}_i$ and $F^{-1}_n=d^{[1]}_j\otimes\cdots\otimes d^{[n]}_j$.\\
\medskip
\noindent$(1)$ $F_{m+n}=(F_m\otimes F_n)(\Delta^m\otimes\Delta^n)(F)$.\\
\medskip
% \noindent$(2)$ $F^{-1}_{m+n}=(\Delta^m\otimes\Delta^n)(F)(F^{-1}_m\otimes F^{-1}_n)$.\\
\noindent$(2)$ $f^{[1]}_iS(f^{[m]}_i)_{(1)}\otimes\cdots\otimes f^{[m-1]}_iS(f^{[m]}_i)_{(m-1)}=uS(d^{[m-1]}_i)\otimes\cdots\otimes uS(d^{[1]}_i)$.\\
\medskip
\noindent$(3)$ $S(d^{[1]}_j)_{(1)}d^{[2]}_j\otimes\cdots\otimes S(d^{[1]}_j)_{(m-1)}d^{[m]}_j=S(f^{[m-1]}_j)u^{-1}\otimes\cdots\otimes S(f^{[1]}_j)u^{-1}$.\\
\medskip
\noindent$(4)$ $\left(\bigotimes\limits_{s=1}^{m-1}f^{[s]}_i\right)\otimes S(f^{[m]}_i)u^{-1}f^{[m+1]}_i\otimes\left(\bigotimes\limits_{s=m+2}^{n}f^{[s]}_i\right)=\left(\bigotimes\limits_{s=1}^{m-1}f^{[s]}_i\right)\otimes 1\otimes\left(\bigotimes\limits_{s=m}^{n-2}f^{[s]}_i\right)$.\\
\medskip
\noindent$(5)$ $\left(\bigotimes\limits_{s=1}^{m-1}d^{[s]}_j\right)\otimes d^{[m]}_juS(d^{[m+1]}_j)\otimes\left(\bigotimes\limits_{s=m+2}^{n}d^{[s]}_j\right)=\left(\bigotimes\limits_{s=1}^{m-1}d^{[s]}_j\right)\otimes 1\otimes\left(\bigotimes\limits_{s=m}^{n-2}d^{[s]}_j\right)$.\\
\medskip
\noindent$(6)$
$\Delta^n(u)=(d^{[1]}_i\otimes\cdots\otimes d^{[n]}_i)(u\otimes\cdots\otimes u)(S(d^{[n]}_j)\otimes\cdots\otimes S(d^{[1]}_j))$.\\
\medskip
\noindent$(7)$
$\Delta^n(u^{-1})=(S(f^{[n]}_i)\otimes\cdots\otimes S(f^{[1]}_i))(u^{-1}\otimes\cdots\otimes u^{-1})(f^{[1]}_j\otimes\cdots\otimes f^{[n]}_j)$.\\
\medskip
\noindent$(8)$ 
$\Delta^n(Q)=(d^{[1]}_i\otimes\cdots\otimes d^{[n]}_i)(Q\otimes\cdots\otimes Q)(S^2(f^{[n]}_i)\otimes\cdots\otimes S^2(f^{[n]}_i))$.\\
\medskip
\noindent$(9)$ We have
\[\begin{aligned}
&\lambda(S(\Lambda_{(2)})S_F\circ Y\circ\Delta_F^{n-1}(\Lambda_{(1)}))\\
=&\lambda S (d^{[1]}_jY(f^{[1]}_i\Lambda_{(1)}d^{[2]}_j\otimes\cdots\otimes f^{[n-1]}_i\Lambda_{(n-1)}d^{[n]}_j)f^{[n]}_i\Lambda_{(n)}).
\end{aligned}\]
for any linear map $Y:H^{\otimes(n-1)}\rightarrow H$. \qed
\end{prop}

\section{The Kuperberg invariant}

Given a finite-dimensional Hopf algebra $H$, Kuperberg defined an topological invariant for any closed 3-manifold with a framing \cite{Kup96}. When 3-manifold $M$ is equipped with Riemannian metric, a framing $f$ on $M$ consists of three orthonormal tangent vector fields $(b_1, b_2, b_3)$ up to homotopy. A classical result in low dimensional topology is that any orientable closed 3-manifold $M$ is obtained by gluing two handle bodies of the same genus along their boundary surfaces \cite{Rol76}. The gluing is described by two groups of disjoint simple closed curves on the surface, one group of them are referred to as lower curves, and the other are called upper curves. The upper and lower curves are also called Heegaard circles. Such orientable closed surface with two groups of disjoint simple closed curves is called a Heegaard diagram of $M$, which gives a diagrammatic presentation of a 3-manifold. 

In this paper, we use planar presentations for Heegaard diagrams. That is, a genus $k$ orientable closed surface minus one point is identified as a plane with $2k$ holes, which are divided into $k$ pairs so that the holes in each pair are connected by handle above the plane. Then the Heegaard circles are represented by curves on the punctured plane whose end points on the holes are understood as being connected through the handles above the plane. Examples are given in the next two sections for the Weeks manifold and 3-torus, which are Heegaard diagrams of genus two and three respectively.

Based on a Heegaard diagram, Kuperberg introduced a combinatorial way to describe a framing $f=(b_1, b_2, b_3)$. It suffices to depict two vector fields in the framing as the third one is determined by right-hand rule. In our planar presentation for Heegaard diagram, the first vector field $b_1$ has $k$ index $-1$ singularities such that each index $-1$ singularity lies on a Heegaard circle and its outward-pointing vectors are tangent to the circle. The index $-1$ singularity is called the base point of the corresponding Heegaard circle. The other vector field $b_2$ is orthogonal to $b_1$ (not necessarily tangent to the surface) except at the base points. Some arcs, called twist fronts, are used to present $b_2$, along which $b_2$ is normal to the surface. A twist front is decorated by small triangles to indicate the direction that $b_2$ rotates by right-hand rule relative to $b_1$. 

When a Heegaard diagram is equipped with such two vector fields $b_1$ and $b_2$, one can count the rotation of Heegaard circles relative to the vector fields. Orient all the Heegaard circles. For any $p$ on a Heegaard circle $c$, Let $\theta_c(p)$ be the counterclockwise rotation of the tangent of $c$ relative to $b_1$ from the base point to $p$ in units of $1=360^\circ$, and $\theta_c$ be the total rotation when it is back to the base point. Note that $\theta_c$ is always a half integer. Let $\phi_c(p)$ be the right-hanged rotation of $b_2$ around $b_1$ from the base point to $p$, and $\phi_c$ be the total rotation. Then $\phi_c(p)$ is the aggregated number of signed crossings of $c$ with twist fronts. The sign is positive if the orientation of $c$ coincide with the triangle of the twist front. And the base point of $c$ contributes half of the sign to the total rotation $\phi_c$. The vector fields $b_1$ and $b_2$ are said to be admissible if they satisfies the following conditions:
\[
\begin{aligned}
	\theta_\mu=-\phi_\mu~~~~\text{for each upper curve}~~~~ \mu,~~~~ 
	\text{and}~~~~\theta_\eta=\phi_\eta~~~~\text{for each lower curve}~~~~\eta.
\end{aligned}
\]  
A Heegaard diagram with two admissible vector fields is called a framed Heegaard diagram. It is shown in \cite{Kup96} that a framed Heegaard diagram can completely describe a framed 3-manifold.   

Given a framed Heegaard diagram of a framed 3-manifold with lower and upper curves $\{\eta_1, \ldots, \eta_k\}$ and $\{\mu_1, \ldots, \mu_k\}$ respectively, and a finite-dimensional Hopf algebra $H$ with a pair of normalized integrals $\Lambda$ and $\lambda$, the Kuperberg invariant is constructed by combining the geometric information and the Hopf algebra operations. Define $T: H\rightarrow H$ be the linear map $T(x)=g^{-1}S^2(x)g$, where $g\in H$ is the distinguished group-like element.  Let  $p$ be an intersection point between $\eta_i$ and $\mu_j$ , we define $s(p)$ and $t(p)$ by
\[\begin{aligned}
s(p)=2(\theta_{\eta_i}(p)-\theta_{\mu_j}(p))+\frac{1}{2},~~~~t(p)=\phi_{\eta_i}(p)-\phi_{\mu_j}(p)
\end{aligned}\]
Note that $s(p)$ is half integer for all intersection points between lower and upper curves. We introduce twisted integrals: for any integer $m$ 
\[\begin{aligned}
\Lambda_{m-\frac{1}{2}}=\alpha^{-m}\rightharpoonup S(\Lambda),~~~~\lambda_{m-\frac{1}{2}}=g^m\rightharpoonup \lambda.
\end{aligned}\]
where the first action $\rightharpoonup$ is that $\alpha^{-m}\rightharpoonup x=x_{(1)}\alpha^{-m}(x_{(2)})$ for $x\in H$.

Let $I$ be the set of all intersection points between the lower and upper curves, and $n=|I|$, $n_i=|\eta_i\cap I|$.  For each lower curve $\eta_i$, we write $L^i=\Lambda_{\theta_{\eta_i}}$ and $\Delta^{n_i}(L^i)=L^i_{(1)}\otimes \cdots\otimes L^i_{(n_i)}$. Then the Kuperberg invariant can be constructed in the following steps:

\noindent (1) For each lower curve $\eta_i$, assign the coproduct factor $L^i_{(m)}$ of the $L^i$ to the intersection points following the direction of $\eta_i$ starting from its base point. 

\noindent (2) Label the intersection point $p$ by $S^{s(p)}T^{t(p)}(L^i_{(m)})$ when $L^i_{(m)}$ is assigned to $p$.

\noindent (3) For each upper curve $\mu_j$, multiply the labels on $\mu_j$ starting from the base point and following the direction of $\mu_j$. Then we have a product $w_j\in H$ of $S^{s(p)}T^{t(p)}(L^i_{(m)})$'s for the upper curve $\mu_j$.  

\noindent (4) The Kuperberg invariant is given by
\[\begin{aligned}
Z(M, f, H)=\prod\limits_{\mu_j}\langle \lambda_{-\theta(\mu_j)}, w_j\rangle
\end{aligned}\]
Note that since the coproduct of $L^i$ is a sum of simple tensors, $Z(M, f, H)$ is a summation of these product for all summands of $L^i$'s.

\section{The Weeks manifold}

In this section, we consider the Kuperberg invariant of the Weeks manifold, which is the hyperbolic closed 3-manifold with the smallest volume. The Weeks manifold admits a genus two Heegaard diagram as in Figure \ref{weeks} \cite{CK20}. 

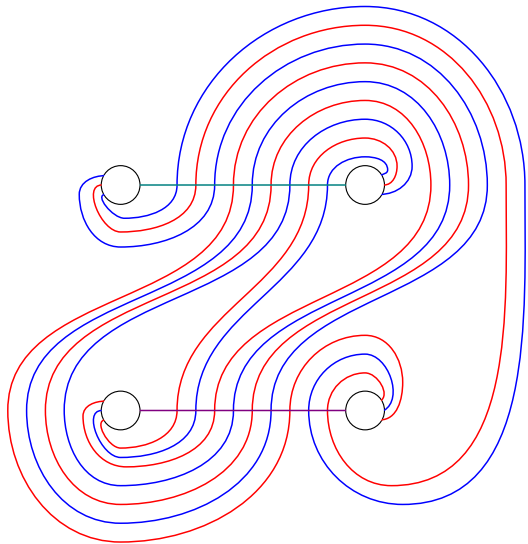
\begin{figure}[htbp]
	\centering
	\begin{tikzpicture}[scale=0.25,decoration={markings,
		mark=between positions 0 and 0.9 step 10mm with {\draw[->] (0.1, 0)--(0.2, 0);}}]
	\draw [line width=0.3mm] (0, 12) circle(1);
	\draw [line width=0.3mm] (13, 12) circle(1);
	\draw [line width=0.3mm] (0, 0) circle(1);
	\draw [line width=0.3mm] (13, 0) circle(1);
	\draw [red, line width=0.2mm] (-1, 12) to [out=180, in=180] (0, 9.5) to [out=0, in=-90] (4, 12) to [out=90, in=180] (13, 20.5) to [out=0, in=90] (20.5, 12) to [out=-90, in=0] (14.4, -4) to [out=180, in=-90] (11,0) to [out=90, in=180] (13, 2) to [out=0, in=30] (13.8, 0.5) ;
	\draw [red, line width=0.2mm] (-0.8, 0.5) to [out=-180, in=90] (-2, -0.3)to [out=-90, in=180] (0.3, -3) to [out=0, in=-90] (5, 0) to [out=90, in=-90] (16.5, 12) to [out=90, in=0] (13, 16.5) to [out=180, in=90] (8,12) to [out=-90, in=90] (-4, 0) to [out=-90, in=180] (0, -5) to [out=0, in=-90] (7,0) to [out=90, in=-90] (18.5, 12) to [out=90, in=0] (13, 18.5) to [out=180, in=90] (6, 12) to [out=-90, in=90] (-6, 0) to [out=-90, in=180] (0,-7) to [out=0, in=-90] (9, 0) to [out=90, in=180] (13, 4) to [out=0, in=90] (15,1.5)  to [out=-90, in=0] (13.8, -0.5) ;
	\draw [red, line width=0.2mm] (-0.8, -0.5) to [out=180, in=180] (0, -2)  to [out=0, in=-90] (3,0) to [out=90, in=-90] (10,12) to [out=90, in=180] (13, 14.5) to [out=0, in=90] (14.7,13)to [out=-90, in=0] (14,12) ;
	\draw [blue, line width=0.2mm] (-1, 0) to [out=180, in=180] (0, -2.5) to [out=0, in=-90] (4, 0) to [out=90, in=-90] (11,12) to [out=90, in=180] (13, 13.5) to [out=0, in=90] (14.2,12.8)to [out=-90, in=30] (13.8,12.5)  ;
	\draw [blue, line width=0.2mm] (-0.8, 11.5) to [out=180, in=180] (0.2, 10.23) to [out=0, in=-90] (3, 12) to [out=90, in=180] (13, 21.5) to [out=0, in=90] (21.5, 12) to [out=-90, in=0] (15, -5) to [out=180, in=-90] (10,0) to [out=90, in=180] (13, 3) to [out=0, in=30] (14, 0) ;
	\draw [blue, line width=0.2mm] (-0.8, 12.5) to [out=-180, in=90] (-2.2, 11.5)to [out=-90, in=180] (0, 8.7) to [out=0, in=-90] (5, 12) to [out=90, in=180] (13, 19.5) to [out=0, in=90] (19.5, 12) to [out=-90, in=90] (8,0) to [out=-90, in=0] (0,-6) to [out=180, in=-90] (-5, 0) to [out=90, in=-90] (7,12) to [out=90, in=180] (13, 17.5) to [out=0, in=90] (17.5, 12) to [out=-90, in=90] (6, 0) to [out=-90, in=0] (0,-4) to [out=180, in=-90] (-3,0) to [out=90, in=-90] (9, 12) to [out=90, in=180] (13, 15.5) to [out=0, in=90] (15.5,13)  to [out=-90, in=0] (13.8, 11.5) ;
	\draw [teal, line width=0.2mm] (1, 12) to  (12, 12);
	\draw [violet, line width=0.2mm] (1, 0) to (12, 0);
	\fill [white] (0, 12) circle(1);
	\fill [white] (13, 12) circle(1);
	\fill [white] (0, 0) circle(1);
	\fill [white] (13, 0) circle(1);
	\end{tikzpicture}
	\caption{Heegaard diagram of the Weeks manifold}
	\label{weeks}
\end{figure}

In this Heegaard diagram, the each pair of horizontal holes are connected by a handle. So the blue arcs are joined through these two handles and are parts of a single simple closed curve $\mu_1$. Similarly, the red arcs also form a simple closed curve $\mu_2$. They are referred to as upper curves. The two horizontal lines also give two simple closed curves $\eta_1$ (yellow), $\eta_2$ (green) and are referred to as lower curves. 

\begin{figure}[htbp]
	\centering
	\scalebox{0.8}{\begin{tikzpicture}[scale=0.55,decoration={markings,
		mark=between positions 0 and 0.9 step 10mm with {\draw[->] (0.1, 0)--(0.2, 0);}}]
	\draw [line width=0.3mm] (0, 12) circle(1);
	\draw [line width=0.3mm] (15, 12) circle(1);
	\draw [line width=0.3mm] (0, 0) circle(1);
	\draw [line width=0.3mm] (15, 0) circle(1);
	\draw [red, line width=0.2mm] (-1, 12) to [out=180, in=90] (-2, 10) to [out=-90, in=180] (0, 8) to [out=0, in=-90] (8, 12) to [out=90, in=180] (15, 21) to [out=0, in=90] (23, 12) to [out=-90, in=0] (15, -4) to [out=180, in=-90] (10,0) to [out=90, in=180] (15, 2) to [out=0, in=30] (15.8, 0.5) ;
	\draw [red, line width=0.2mm] (-0.8, 0.5) to [out=-180, in=90] (-2, -0.3)to [out=-90, in=180] (0.3, -4) to [out=0, in=-90] (4, 0) to [out=90, in=-90] (19, 12) to [out=90, in=0] (15, 17) to [out=180, in=90] (12,12) to [out=-90, in=90] (-4, 0) to [out=-90, in=180] (0, -6) to [out=0, in=-90] (6,0) to [out=90, in=-90] (21, 12) to [out=90, in=0] (15, 19) to [out=180, in=90] (10, 12) to [out=-90, in=90] (-6, 0) to [out=-90, in=180] (0,-8) to [out=0, in=-90] (8, 0) to [out=90, in=180] (15, 4) to [out=0, in=90] (17,2)  to [out=-90, in=0] (15.8, -0.5) ;
	\draw [red, line width=0.2mm] (-0.8, -0.5) to [out=180, in=180] (0, -2)  to [out=0, in=0] (0.8,-0.5) ;
	\draw [red, line width=0.2mm] (14.2, -0.5) to [out=180, in=-90] (11, 0)  to [out=90, in=180] (14.2,0.5) ;
	\draw [red, line width=0.2mm] (0.8, 0.5) to [out=45, in=-160] (7.5, 6)to [out=20, in=-120] (14.2, 11.5);
	\draw [red, line width=0.2mm] (0.8, 11.5) to [out=-40, in=-90] (5, 12)to [out=90, in=40] (0.8, 12.5);
	\draw [red, line width=0.2mm] (14.2, 12.5) to [out=150, in=180] (15, 15)to [out=0, in=90] (17, 13)to [out=-90, in=0] (16, 12);	
	\draw [blue, line width=0.2mm] (-0.8, 11.5) to [out=180, in=180] (0, 9) to [out=0, in=-90] (7, 12) to [out=90, in=180] (15, 22) to [out=0, in=90] (24, 12) to [out=-90, in=0] (15, -5) to [out=180, in=-90] (9,0) to [out=90, in=180] (15, 3) to [out=0, in=30] (16, 0) ;
	\draw [blue, line width=0.2mm] (-1, 0) to [out=180, in=180] (0, -3) to [out=0, in=-90] (3, -0.5)  to [out=90, in=-90] (2.7, -0.2) to [out=90, in=-160] (3, 0) to [out=20, in=-90] (3.3, 0.2) to [out=90, in=-90] (3, 0.5)to [out=90, in=-100] (14.3,11.2);
	\draw [blue, line width=0.2mm] (0.7, 11.2) to [out=-50, in=-90] (6, 12)to [out=90, in=50] (0.65, 12.8);
	\draw [blue, line width=0.2mm] (14.35, 12.8) to [out=150, in=180] (15, 14)to [out=0, in=0] (15.8, 12.5);
	\draw [blue, line width=0.2mm] (-0.8, 12.5) to [out=-180, in=90] (-3, 10)to [out=-90, in=180] (0, 7) to [out=0, in=-90] (9, 12) to [out=90, in=180] (15, 20) to [out=0, in=90] (22, 12) to [out=-90, in=90] (7,0) to [out=-90, in=0] (0,-7) to [out=180, in=-90] (-5, 0) to [out=90, in=-90] (11,12) to [out=90, in=180] (15, 18) to [out=0, in=90] (20, 12) to [out=-90, in=90] (5, 0) to [out=-90, in=0] (0,-5) to [out=180, in=-90] (-3,0) to [out=90, in=-160] (8,6.8) to [out=20, in=-135] (14.2, 11.8) ;
	\draw [blue, line width=0.2mm] (0.8, 11.8) to [out=-30, in=-90] (4, 12)to [out=90, in=30] (0.8, 12.3);
	\draw [blue, line width=0.2mm] (14.2, 12.3) to [out=150, in=180] (15, 16)to [out=0, in=90] (18, 13)to [out=-90, in=0] (15.8, 11.5);	
	\draw [teal, line width=0.2mm] (1, 12) to (11.5, 12) to [out=0, in=180] (12, 11.5)  to [out=0, in=180] (12.5, 12) to (14, 12);
	\draw [violet, line width=0.2mm] (1, 0) to (2.3, 0) to [out=0, in=180] (3, -0.7)  to [out=0, in=180] (3.7, 0) to (11.5,0) to [out=0, in=180] (11.75, 0.3)to [out=0, in=90] (12, 0) to [out=-90, in=180] (12.25, -0.3) to [out=0, in=180] (12.5, 0)to (14, 0);
	\draw [dashed, line width=0.2mm, postaction={decorate}] (3, 0) to (0, 0);
	\draw [dashed, line width=0.2mm, postaction={decorate}] (3, 0) to (12, 0);
	\draw [dashed, line width=0.2mm, postaction={decorate}] (14, 0) to (12, 0);
	\draw [dashed, line width=0.2mm, postaction={decorate}] (3, -9) to  (3, 0);	
	\draw [dashed, line width=0.2mm, postaction={decorate}] (-7, 0) to (0, 0);	
	\draw [dashed, line width=0.2mm, postaction={decorate}] (-7, -7) to (0, 0);	
	\draw [dashed, line width=0.2mm, postaction={decorate}] (12, 0) to  (12,-9);	
	\draw [dashed, line width=0.2mm, postaction={decorate}] (14, 0) to (25, 0);			
	\draw [dashed, line width=0.2mm, postaction={decorate}] (3, 12) to (0, 12);
	\draw [dashed, line width=0.2mm, postaction={decorate}] (3, 12) to (12, 12);
	\draw [dashed, line width=0.2mm, postaction={decorate}] (1, 12) to (12, 12);
	\draw [dashed, line width=0.2mm, postaction={decorate}] (3, 22) to (3, 12);	
	\draw [dashed, line width=0.2mm, postaction={decorate}] (-7, 7.5) to [out=0, in=-90] (3, 12);	
	\draw [dashed, line width=0.2mm, postaction={decorate}] (2.5, 22) to [out=-90, in=60] (0, 12);	
	\draw [dashed, line width=0.2mm, postaction={decorate}] (2.5, -9) to [out=90, in=-60] (0, 0);
	\draw [dashed, line width=0.2mm, postaction={decorate}] (-7, 19) to (0, 12);	
	\draw [dashed, line width=0.2mm, postaction={decorate}] (-7, 12) to (0, 12);	
	\draw [dashed, line width=0.2mm, postaction={decorate}] (-7, 9) to (-2, 9) to [out=0, in=-90] (1, 11) to [out=90, in=-30] (0, 12);	
	\draw [dashed, line width=0.2mm, postaction={decorate}] (12, 12) to (12, 22);	
	\draw [dashed, line width=0.2mm, postaction={decorate}] (12, 12) to [out=-90, in=180] (25, 7.5);	
	\draw [dashed, line width=0.2mm, postaction={decorate}] (15, 12) to [out=150, in=-90] (12.5, 21.5);	
	\draw [dashed, line width=0.2mm, postaction={decorate}] (15, 0) to [out=150, in=-90] (12.5, -8);	
	\draw [dashed, line width=0.2mm, postaction={decorate}] (15, 12) to (25, 22);	
	\draw [dashed, line width=0.2mm, postaction={decorate}] (15, 0) to (25, -9);
	\draw [dashed, line width=0.2mm, postaction={decorate}] (15, 12) to (25, 12);	
	\draw [dashed, line width=0.2mm, postaction={decorate}] (15, 12) to [out=-150, in=90] (13.5, 10) to [out=-90, in=180] (15, 9) to (25, 9);	
	\draw [dashed, line width=0.2mm, postaction={decorate}] (-7, 6) to (25, 6);	
	\draw [dashed, line width=0.2mm, postaction={decorate}] (-7, 4.5) to [out=0, in=90] (3, 0);	
	\draw [dashed, line width=0.2mm, postaction={decorate}] (-7, 0) to (0, 0);	
	\draw [dashed, line width=0.2mm, postaction={decorate}] (-7, 3) to (-2, 3) to [out=0, in=90] (1, 1) to [out=-90, in=30] (0, 0);	
	\draw [dashed, line width=0.2mm, postaction={decorate}] (12, 0) to [out=90, in=180] (25, 4.5);	
	\draw [dashed, line width=0.2mm, postaction={decorate}] (15, 0) to (25, 0);	
	\draw [dashed, line width=0.2mm, postaction={decorate}] (15, 0) to [out=120, in=-90] (13.5, 2) to [out=90, in=180] (15, 3) to (25, 3);	
	\draw [dashed, line width=0.2mm, postaction={decorate}] (3.3, 19) to [out=-90, in=180] (7.5, 14) to [out=0, in=-90] (11.7, 19);
	\draw [dashed, line width=0.2mm, postaction={decorate}] (3.3, -7) to [out=90, in=180] (7.5, -2) to [out=0, in=90] (11.7, -7);
	\draw [dashed, line width=0.2mm, postaction={decorate}] (0.5, 7) to [out=0, in=180] (7.5, 8) to [out=0, in=180] (16.5, 7);
	\draw [dashed, line width=0.2mm, postaction={decorate}] (0.5, 5) to [out=0, in=180] (7.5, 4) to [out=0, in=180] (16.5, 5);
	\draw [decorate, decoration={zigzag, segment length=5pt, amplitude=2pt}] (3, 12.1) to [out=135, in=15] (1, 12.19);%twist front
	\draw (3, 12) to [out=135, in=15] (1, 12.1);
	\draw [decorate, decoration={zigzag, segment length=5pt, amplitude=2pt}] (2.8, 0) to [out=120, in=60] (1, 0.05);%twist front
	\draw (3, 0) to [out=120, in=60] (1, 0.2);
	\draw [decorate, decoration={zigzag, segment length=5pt, amplitude=2pt}] (12, 12.15) to [out=60, in=120] (14, 12.2);%twist front
	\draw (12,12) to [out=60, in=120] (14, 12.1);
	\draw [decorate, decoration={zigzag, segment length=5pt, amplitude=2pt}] (12.15, 0) to [out=60, in=120] (13.95, 0);%twist front
	\draw (12,0) to [out=60, in=120] (14, 0.2);
	\node[font=\tiny] at (4+0.3,12-0.3) {$p_1$};
	\node[font=\tiny] at (5+0.3,12-0.3) {$p_2$};
	\node[font=\tiny] at (6+0.3,12-0.3) {$p_3$};
	\node[font=\tiny] at (7+0.3,12-0.3) {$p_4$};
	\node[font=\tiny] at (8+0.3,12-0.3) {$p_5$};
	\node[font=\tiny] at (9+0.3,12-0.3) {$p_6$};
	\node[font=\tiny] at (10+0.3,12-0.3) {$p_7$};
	\node[font=\tiny] at (11+0.3,12-0.3) {$p_8$};
	\node[font=\tiny] at (12+0.3,12-0.7) {$p_9$};
	\node[font=\tiny] at (3+0.3,0-0.8) {$q_1$};
	\node[font=\tiny] at (4+0.3,0-0.3) {$q_2$};
	\node[font=\tiny] at (5+0.3,0-0.3) {$q_3$};
	\node[font=\tiny] at (6+0.3,0-0.3) {$q_4$};
	\node[font=\tiny] at (7+0.3,0-0.3) {$q_5$};
	\node[font=\tiny] at (8+0.3,0-0.3) {$q_6$};
	\node[font=\tiny] at (9+0.3,0-0.3) {$q_7$};
	\node[font=\tiny] at (10+0.3,0-0.3) {$q_8$};
	\node[font=\tiny] at (11+0.3,0-0.3) {$q_9$};
	\fill [white] (0, 12) circle(1);
	\fill [white] (15, 12) circle(1);
	\fill [white] (0, 0) circle(1);
	\fill [white] (15, 0) circle(1);
	\end{tikzpicture}}
	\caption{Framed Heegaard diagram of the Weeks manifold}
	\label{framed_weeks}
\end{figure}
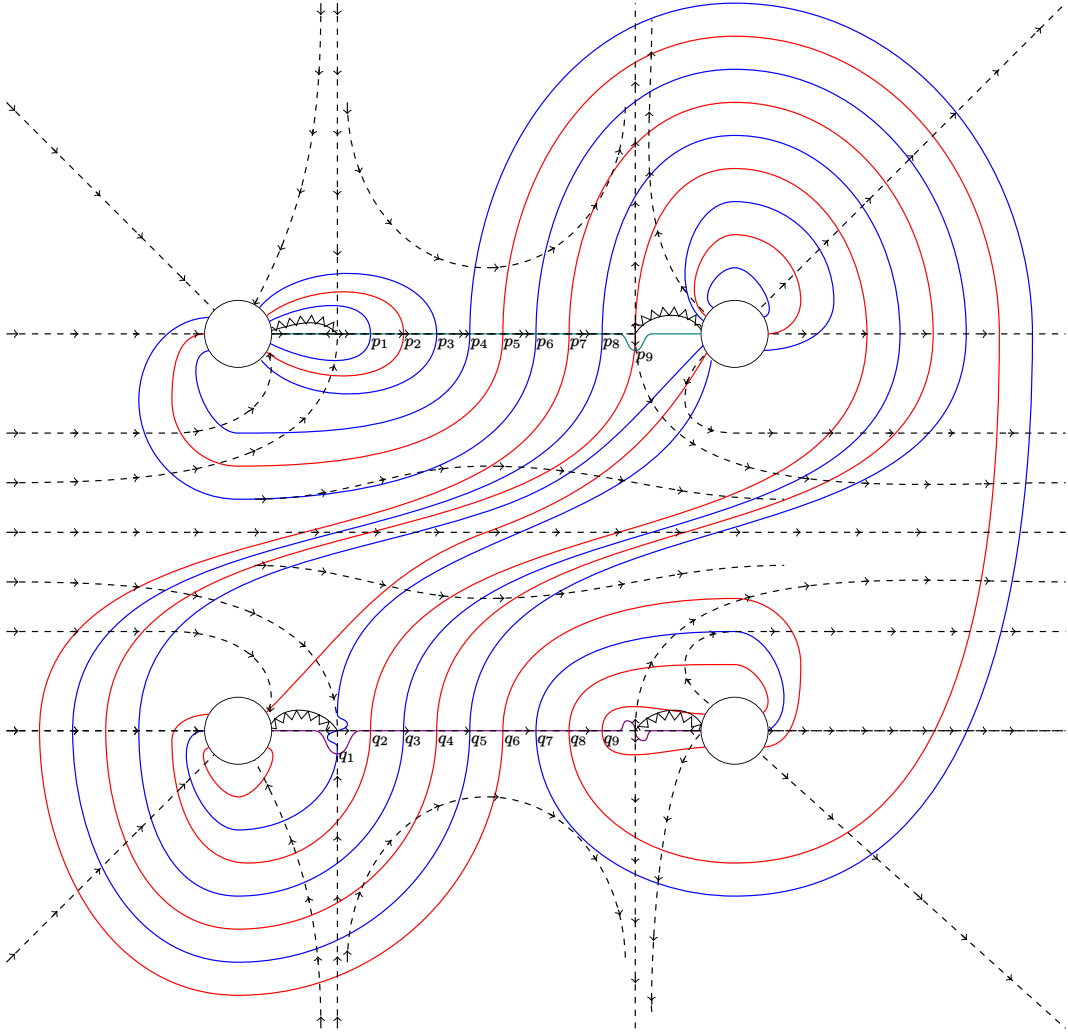

We set up a diagram framing $f_0$ for this Heegaard diagram as shown in Figure \ref{framed_weeks}. On this framed diagram, the upper curves in Figure \ref{weeks} have been deformed by isotopy moves that does not change the homeomorphism type of the 3-manifold. Starting from  their based points, the upper curves run towards to the right and the lower curves go downwards. In the following table we record the angles between the tangent of each circle and the vector fields, and the power of antipode $S$ labeled at each intersection points.  
Since all the Heegaard circles do not intersect with the twist fronts except at the base points, we only have the action of $S$ at each intersection point and the total rotations are: $\theta_{\eta_1}=\phi_{\eta_1}=\frac{1}{2}$, $\theta_{\eta_2}=\phi_{\eta_2}=\frac{1}{2}$, $\theta_{\mu_1}=-\phi_{\mu_1}=-\frac{1}{2}$ and $\theta_{\mu_2}=-\phi_{\mu_2}=\frac{1}{2}$.

\[
\begin{tabular}{|c|c|c|c|c|c|c|c|c|c|c|} 
	\hline 
	 & ~$p_1$~  & ~$p_2$~ & ~$p_3$~ & ~$p_4$~ & ~$p_5$~  & ~$p_6$~ & ~$p_7$~ & ~$p_8$~ & ~$p_9$~ \tabularnewline
	\hline 
	$\theta_\eta$ & ~$0$~  & ~$0$~ & ~$0$~ & ~$0$~ & ~$0$~  & ~$0$~ & ~$0$~ & ~$0$~ & ~$\frac{1}{4}$~   \tabularnewline
	\hline 
	$\theta_\mu$ & ~$\frac{3}{4}$~  & ~$\frac{1}{4}$~ & ~$\frac{3}{4}$~ & ~$\frac{7}{4}$~ & ~$\frac{5}{4}$~  & ~$\frac{7}{4}$~ & ~$\frac{3}{4}$~ & ~$\frac{3}{4}$~ & ~$0$~  \tabularnewline
	\hline 
	& ~$S^{-1}$~  & ~$S^0$~ & ~$S^{-1}$~ & ~$S^{-3}$~ & ~$S^{-2}$~  & ~$S^{-3}$~ & ~$S^{-1}$~ & ~$S^{-1}$~ & ~$S^1$~  \tabularnewline
	\hline
\end{tabular}
\]

\[
\begin{tabular}{|c|c|c|c|c|c|c|c|c|c|c|} 
	\hline 
	& ~$q_1$~  & ~$q_2$~ & ~$q_3$~ & ~$q_4$~ & ~$q_5$~  & ~$q_6$~ & ~$q_7$~ & ~$q_8$~ & ~$q_9$~ \tabularnewline
	\hline 
	$\theta_\eta$ & ~$-\frac{1}{4}$~  & ~$0$~ & ~$0$~ & ~$0$~ & ~$0$~  & ~$0$~ & ~$0$~ & ~$0$~ & ~$0$~ \tabularnewline
	\hline 
	$\theta_\mu$ & ~$\frac{1}{2}$~  & ~$\frac{1}{4}$~ & ~$\frac{1}{4}$~ & ~$\frac{1}{4}$~ & ~$\frac{5}{4}$~  & ~$\frac{5}{4}$~ & ~$\frac{3}{4}$~ & ~$\frac{1}{4}$~ & ~$\frac{1}{4}$~ \tabularnewline
	\hline 
	& ~$S^{-1}$~  & ~$S^0$~ & ~$S^0$~ & ~$S^0$~ & ~$S^{-2}$~  & ~$S^{-2}$~ & ~$S^{-1}$~ & ~$S^0$~ & ~$S^0$~ \tabularnewline
	\hline
\end{tabular}
\]

The order of intersection points on $\mu_1$ is: $q_1$, $q_7$, $p_4$, $p_1$, $q_3$, $p_8$, $q_5$, $p_6$ and $p_3$. And the order on $\mu_2$ is: $p_9$, $q_4$, $p_7$, $q_6$, $q_9$, $p_2$, $p_5$, $q_8$ and $q_2$. By the above data of rotation, the corresponding Kuperberg invariant is written as   
\[
\begin{aligned}
Z(W, f_0, H)=&\lambda S^{-1}(S^{-1}(\Lambda^2_{(1)})S^{-1}(\Lambda^2_{(7)})S^{-3}(\Lambda^1_{(4)})S^{-1}(\Lambda^1_{(1)})\Lambda^2_{(3)}S^{-1}(\Lambda^1_{(8)})S^{-2}(\Lambda^2_{(5)})S^{-3}(\Lambda^1_{(6)})\\
&S^{-1}(\Lambda^1_{(3)}))\lambda(S(\Lambda^1_{(9)})\Lambda^2_{(4)}S^{-1}(\Lambda^1_{(7)})S^{-2}(\Lambda^2_{(6)})\Lambda^2_{(9)}\Lambda^1_{(2)}S^{-2}(\Lambda^1_{(5)})\Lambda^2_{(8)}\Lambda^2_{(2)})\\
=&\lambda S^{-2}(\Lambda^1_{(3)}S^{-2}(\Lambda^1_{(6)})S^{-1}(\Lambda^2_{(5)})\Lambda^1_{(8)}S(\Lambda^2_{(3)})\Lambda^1_{(1)}S^{-2}(\Lambda^1_{(4)})\Lambda^2_{(7)}\Lambda^2_{(1)})\\
&\lambda(S(\Lambda^1_{(9)})\Lambda^2_{(4)}S^{-1}(\Lambda^1_{(7)})S^{-2}(\Lambda^2_{(6)})\Lambda^2_{(9)}\Lambda^1_{(2)}S^{-2}(\Lambda^1_{(5)})\Lambda^2_{(8)}\Lambda^2_{(2)})\\
=&\alpha(g^{-1})\lambda (\Lambda^1_{(3)}S^{-2}(\Lambda^1_{(6)})S^{-1}(\Lambda^2_{(5)})\Lambda^1_{(8)}S(\Lambda^2_{(3)})\Lambda^1_{(1)}S^{-2}(\Lambda^1_{(4)})\Lambda^2_{(7)}\Lambda^2_{(1)})\\
&\lambda(S(\Lambda^1_{(9)})\Lambda^2_{(4)}S^{-1}(\Lambda^1_{(7)})S^{-2}(\Lambda^2_{(6)})\Lambda^2_{(9)}\Lambda^1_{(2)}S^{-2}(\Lambda^1_{(5)})\Lambda^2_{(8)}\Lambda^2_{(2)})
\end{aligned}
\]

Since $\alpha(g)$ is a gauge invariant \cite{Shi15}, we just need to verify the gauge invariance for $Z(H):=\alpha(g)Z(W, f_0, H)$.
\begin{equation}\label{eq:ZH}
\begin{aligned}
	Z(H)=&\lambda (\Lambda^1_{(3)}S^{-2}(\Lambda^1_{(6)})S^{-1}(\Lambda^2_{(5)})\Lambda^1_{(8)}S(\Lambda^2_{(3)})\Lambda^1_{(1)}S^{-2}(\Lambda^1_{(4)})\Lambda^2_{(7)}\Lambda^2_{(1)})\\
	&\lambda(S(\Lambda^1_{(9)})\Lambda^2_{(4)}S^{-1}(\Lambda^1_{(7)})S^{-2}(\Lambda^2_{(6)})\Lambda^2_{(9)}\Lambda^1_{(2)}S^{-2}(\Lambda^1_{(5)})\Lambda^2_{(8)}\Lambda^2_{(2)})
\end{aligned}
\end{equation}
Note that for a linear map $X\in\mathrm{End}(H)$, its trace is computed in terms of integrals by Theorem \ref{integral}(2). 
Then $Z$ can written as a trace of some linear map on $H\otimes H$. 
\[
\begin{aligned}
Z(H)=\mathrm{Tr}((S\otimes S)\circ P_2),
\end{aligned}
\]
where $P_2\in \mathrm{End}(H\otimes H)$ is
\[
\begin{aligned}
	P_2(x\otimes y)=&S^{-2}(x_{(3)})S^{-4}(x_{(6)})S^{-3}(y_{(4)})S^{-2}(x_{(8)})S^{-1}(y_{(2)})S^{-2}(x_{(1)})S^{-4}(x_{(4)})S^{-2}(y_{(6)})\\
	&\otimes S^{-1}(y_{(1)})S^{-1}(x_{(7)})S^{-3}(x_{(5)})S^{-1}(x_{(2)})S^{-1}(y_{(8)})S^{-3}(y_{(5)})S^{-2}(x_{(7)})S^{-1}(y_{(3)}).
\end{aligned}
\]

Since $\Lambda^1\otimes\Lambda^2$ and $\lambda\otimes \lambda$ is a pair of normalized integrals of $H\otimes H$ and its dual, we can rewrite $Z$ in the following form using the trace formula:
\[
\begin{aligned}
	Z(H)=&\lambda S (S^{-1}(\Lambda^2_{(1)})S^{-1}(\Lambda^2_{(7)})S^{-3}(\Lambda^1_{(5)})S^{-1}(\Lambda^1_{(2)})S^{-1}(\Lambda^2_{(8)})S^{-3}(\Lambda^2_{(5)})S^{-2}(\Lambda^1_{(7)})S^{-1}(\Lambda^2_{(3)})\Lambda^1_{(9)})\\
	&\lambda S(S^{-2}(\Lambda^1_{(3)})S^{-4}(\Lambda^1_{(6)})S^{-3}(\Lambda^2_{(4)})S^{-2}(\Lambda^1_{(8)})S^{-1}(\Lambda^2_{(2)})S^{-2}(\Lambda^1_{(1)})S^{-4}(\Lambda^1_{(4)})S^{-2}(\Lambda^2_{(6)})\Lambda^2_{(9)})\,.
\end{aligned}
\]

In the following, we show that $Z(H)$ remains unchanged under gauge transformation. In fact, after gauge transformation, by Theorem \ref{cocycle}(10), we have that

\[
\begin{aligned}
	Z_F:=&Z(H_F)=\mathrm{Tr}((S_F\otimes S_F)\circ P_{2F})=\lambda(S(\Lambda^1_{(2)}\otimes\Lambda^2_{(2)})(S_F\otimes S_F)\circ P_{2F}(\Lambda^1_{(1)}\otimes\Lambda^2_{(1)}))\\
	=&\lambda S (d^{[1]}_iS^{-1}_F(f^{[1]}_p\Lambda^2_{(1)}d^{[2]}_q)S^{-1}_F(f^{[7]}_p\Lambda^2_{(7)}d^{[8]}_q)S^{-3}_F(f^{[5]}_i\Lambda^1_{(5)}d^{[6]}_j)S^{-1}_F(f^{[2]}_i\Lambda^1_{(2)}d^{[3]}_j)S^{-1}_F(f^{[8]}_p\Lambda^2_{(8)}d^{[9]}_q)\\
	&S^{-3}_F(f^{[5]}_p\Lambda^2_{(5)}d^{[6]}_q)S^{-2}_F(f^{[7]}_i\Lambda^1_{(7)}d^{[8]}_j)S^{-1}_F(f^{[3]}_p\Lambda^2_{(3)}d^{[4]}_q)f^{[9]}_i\Lambda^1_{(9)})\\
	&\lambda S(d^{[1]}_qS^{-2}_F(f^{[3]}_i\Lambda^1_{(3)}d^{[4]}_j)S^{-4}_F(f^{[6]}_i\Lambda^1_{(6)}d^{[7]}_j)S^{-3}_F(f^{[4]}_p\Lambda^2_{(4)}d^{[5]}_q)S^{-2}_F(f^{[8]}_i\Lambda^1_{(8)}d^{[9]}_q)S^{-1}_F(f^{[2]}_p\Lambda^2_{(2)}d^{[3]}_q)\\
	&S^{-2}_F(f^{[1]}_i\Lambda^1_{(1)}d^{[2]}_j)S^{-4}_F(f^{[4]}_i\Lambda^1_{(4)}d^{[5]}_j)S^{-2}_F(f^{[2]}_p\Lambda^2_{(6)}d^{[7]}_q)f^{[9]}_p\Lambda^2_{(9)})
\end{aligned}
\]

We apply Theorem \ref{integral}(3) to $\lambda S(d^{[1]}_i\cdots\Lambda^1_{(9)})$ and $S(d^{[1]}_q\cdots\Lambda^2_{(9)})$, then by Proposition \ref{cocycle}(3) and (4) we have the following equality. Here we plug in $S^{-1}_F(x)=S^{-1}(u^{-1}xu)$ as well.

\[
\begin{aligned}
Z_F=&\lambda S (f^{[8]}_pS^{-1}(\Lambda^2_{(1)})d^{[8]}_qf^{[2]}_pS^{-1}(\Lambda^2_{(7)})d^{[2]}_qS^{-2}_F(f^{[4]}_iS^{-1}(\Lambda^1_{(5)})d^{[4]}_j)f^{[7]}_iS^{-1}(\Lambda^1_{(2)})d^{[7]}_jf^{[1]}_pS^{-1}(\Lambda^2_{(8)})d^{[1]}_q\\
&S^{-2}_F(f^{[4]}_pS^{-1}(\Lambda^2_{(5)})d^{[4]}_q)S^{-1}_F(f^{[2]}_iS^{-1}(\Lambda^1_{(7)})d^{[2]}_j)f^{[6]}_pS^{-1}(\Lambda^2_{(3)})d^{[6]}_q\Lambda^1_{(9)})\\
&\lambda S(S^{-1}_F(f^{[6]}_iS^{-1}(\Lambda^1_{(3)})d^{[6]}_j)S^{-3}_F(f^{[3]}_iS^{-1}(\Lambda^{1}_{(6)})d^{[3]}_j)S^{-2}_F(f^{[5]}_pS^{-1}(\Lambda^2_{(4)})d^{[5]}_q)S^{-1}_F(f^{[1]}_iS^{-1}(\Lambda^1_{(8)})d^{[1]}_j)\\
&f^{[7]}_pS^{-1}(\Lambda^2_{(2)})d^{[7]}_qS^{-1}_F(f^{[8]}_iS^{-1}(\Lambda^1_{(1)})d^{[8]}_j)S^{-3}_F(f^{[5]}_iS^{-1}(\Lambda^1_{(4)})d^{[5]}_j)S^{-1}_F(f^{[3]}_pS^{-1}(\Lambda^2_{(6)})d^{[3]}_q)\Lambda^2_{(9)})\\
\end{aligned}
\]

Apply Theorem \ref{integral}(2) to $d^{[6]}_q\Lambda^1_{(9)}$ and by
\[
\begin{aligned}
d^{[1]}_q\otimes\cdots\otimes d^{[5]}_q\otimes d^{[6]}_{q(1)}d^{[1]}_j\otimes\cdots\otimes d^{[6]}_{q(8)}d^{[8]}_j\otimes d^{[7]}_q\otimes\cdots\otimes d^{[8]}_q=d^{[1]}_q \otimes\cdots\otimes d^{[15]}_q,
\end{aligned}
\]
then we get that
\[
\begin{aligned}
	Z_F=&\lambda S (f^{[8]}_pS^{-1}(\Lambda^2_{(1)})d^{[15]}_qf^{[2]}_pS^{-1}(\Lambda^2_{(7)})d^{[2]}_qS^{-2}_F(f^{[4]}_iS^{-1}(\Lambda^1_{(5)})d^{[9]}_q)f^{[7]}_iS^{-1}(\Lambda^1_{(2)})d^{[12]}_qf^{[1]}_pS^{-1}(\Lambda^2_{(8)})d^{[1]}_q\\
	&S^{-2}_F(f^{[4]}_pS^{-1}(\Lambda^2_{(5)})d^{[4]}_q)S^{-1}_F(f^{[2]}_iS^{-1}(\Lambda^1_{(7)})d^{[7]}_q)f^{[6]}_pS^{-1}(\Lambda^2_{(3)})\Lambda^1_{(9)})\\
	&\lambda S(S^{-1}_F(f^{[6]}_iS^{-1}(\Lambda^1_{(3)})d^{[11]}_q)S^{-3}_F(f^{[3]}_iS^{-1}(\Lambda^{1}_{(6)})d^{[8]}_q)S^{-2}_F(f^{[5]}_pS^{-1}(\Lambda^2_{(4)})d^{[5]}_q)S^{-1}_F(f^{[1]}_iS^{-1}(\Lambda^1_{(8)})d^{[6]}_q)\\
	&f^{[7]}_pS^{-1}(\Lambda^2_{(2)})d^{[14]}_qS^{-1}_F(f^{[8]}_iS^{-1}(\Lambda^1_{(1)})d^{[13]}_q)S^{-3}_F(f^{[5]}_iS^{-1}(\Lambda^1_{(4)})d^{[10]}_q)S^{-1}_F(f^{[3]}_pS^{-1}(\Lambda^2_{(6)})d^{[3]}_q)\Lambda^2_{(9)})\\
\end{aligned}
\]
Similarly, by $f^{[1]}_p\otimes\cdots\otimes f^{[7]}_p\otimes f^{[1]}_jf^{[8]}_{p(1)}\otimes\cdots\otimes f^{[8]}_jf^{[8]}_{p(8)}=f^{[1]}_p \otimes\cdots\otimes f^{[15]}_p$,
we get that
\[
\begin{aligned}
	Z_F=&\lambda S (S^{-1}(\Lambda^2_{(1)})d^{[15]}_qf^{[2]}_pS^{-1}(\Lambda^2_{(7)})d^{[2]}_qS^{-2}_F(f^{[11]}_pS^{-1}(\Lambda^1_{(5)})d^{[9]}_q)f^{[14]}_pS^{-1}(\Lambda^1_{(2)})d^{[12]}_qf^{[1]}_pS^{-1}(\Lambda^2_{(8)})d^{[1]}_q\\
	&S^{-2}_F(f^{[4]}_pS^{-1}(\Lambda^2_{(5)})d^{[4]}_q)S^{-1}_F(f^{[9]}_pS^{-1}(\Lambda^1_{(7)})d^{[7]}_q)f^{[6]}_pS^{-1}(\Lambda^2_{(3)})\Lambda^1_{(9)})\\
	&\lambda S(S^{-1}_F(f^{[13]}_pS^{-1}(\Lambda^1_{(3)})d^{[11]}_q)S^{-3}_F(f^{[10]}_pS^{-1}(\Lambda^{1}_{(6)})d^{[8]}_q)S^{-2}_F(f^{[5]}_pS^{-1}(\Lambda^2_{(4)})d^{[5]}_q)S^{-1}_F(f^{[8]}_pS^{-1}(\Lambda^1_{(8)})d^{[6]}_q)\\
	&f^{[7]}_pS^{-1}(\Lambda^2_{(2)})d^{[14]}_qS^{-1}_F(f^{[15]}_pS^{-1}(\Lambda^1_{(1)})d^{[13]}_q)S^{-3}_F(f^{[12]}_pS^{-1}(\Lambda^1_{(4)})d^{[10]}_q)S^{-1}_F(f^{[3]}_pS^{-1}(\Lambda^2_{(6)})d^{[3]}_q)\Lambda^2_{(9)})\\
\end{aligned}
\]
Plug in $S^{-1}_F(x)=S^{-1}(u)S^{-1}(x)S^{-1}(u^{-1})$, $S^{-2}_F(x)=S^{-1}(u)S^{-2}(u^{-1})S^{-2}(x)S^{-2}(u)S^{-1}(u^{-1})$,
\[\begin{aligned}
	S^{-3}_F(x)=S^{-1}(u)S^{-2}(u^{-1})S^{-3}(u)S^{-3}(x)S^{-3}(u^{-1})S^{-2}(u)S^{-1}(u^{-1}).
\end{aligned}\]

\[
\begin{aligned}
Z_F=&\lambda S (S^{-1}(\Lambda^2_{(1)})d^{[15]}_qf^{[2]}_pS^{-1}(\Lambda^2_{(7)})d^{[2]}_qS^{-1}(u)S^{-2}(u^{-1})S^{-2}(f^{[11]}_p)S^{-3}(\Lambda^1_{(5)})S^{-2}(d^{[9]}_q)S^{-1}(u)\\
&S^{-2}(u^{-1})f^{[14]}_pS^{-1}(\Lambda^1_{(2)})d^{[12]}_qf^{[1]}_pS^{-1}(\Lambda^2_{(8)})d^{[1]}_qS^{-1}(u)S^{-2}(u^{-1})S^{-2}(f^{[4]}_p)S^{-3}(\Lambda^2_{(5)})S^{-2}(d^{[4]}_q)\\
&S^{-2}(u)S^{-1}(d^{[7]}_q)S^{-2}(\Lambda^1_{(7)})S^{-1}(f^{[9]}_p)S^{-1}(u^{-1})f^{[6]}_pS^{-1}(\Lambda^2_{(3)})\Lambda^1_{(9)})\\
&\lambda S(S^{-1}(u)S^{-1}(d^{[11]}_q)S^{-2}(\Lambda^1_{(3)})S^{-1}(f^{[13]}_p)S^{-2}(u^{-1})S^{-3}(u)S^{-3}(d^{[8]}_q)S^{-4}(\Lambda^{1}_{(6)})S^{-3}(f^{[10]}_p)\\
&S^{-3}(u^{-1})S^{-2}(f^{[5]}_p)S^{-3}(\Lambda^2_{(4)})S^{-2}(d^{[5]}_q)S^{-2}(u)S^{-1}(d^{[6]}_q)S^{-2}(\Lambda^1_{(8)})S^{-1}(f^{[8]}_p)S^{-1}(u^{-1})f^{[7]}_p\\
&S^{-1}(\Lambda^2_{(2)})d^{[14]}_qS^{-1}(u)S^{-1}(d^{[13]}_q)S^{-2}(\Lambda^1_{(1)})S^{-1}(f^{[15]}_p)S^{-2}(u^{-1})S^{-3}(u)S^{-3}(d^{[10]}_q)S^{-4}(\Lambda^1_{(4)})\\
&S^{-3}(f^{[12]}_p)S^{-3}(u^{-1})S^{-2}(u)S^{-1}(d^{[3]}_q)S^{-2}(\Lambda^2_{(6)})S^{-1}(f^{[3]}_p)S^{-1}(u^{-1})\Lambda^2_{(9)})\,.
\end{aligned}
\]

In order to compare $Z(H)$ with $Z_F$, we need the following lemmas to help us further simplify $Z_F$.

\begin{lemma}\label{lemma1}
For any $x\in H$ and $Y, Z\in \mathrm{End}(H)$, \\
$(1)$ \[\begin{aligned}
&\lambda S (S^{-1}(\Lambda^2_{(1)})Y(\Lambda^2_{(7)},\Lambda^1_{(5)},\Lambda^1_{(2)},\Lambda^2_{(8)},\Lambda^2_{(5)},\Lambda^1_{(7)},\Lambda^2_{(3)})\Lambda^1_{(9)})\\
&\lambda S(xZ(\Lambda^1_{(3)},\Lambda^{1}_{(6)},\Lambda^2_{(4)},\Lambda^1_{(8)},\Lambda^2_{(2)},\Lambda^1_{(1)},\Lambda^1_{(4)},\Lambda^2_{(6)})\Lambda^2_{(9)})\\
=&\lambda S (S^{-1}(\Lambda^2_{(1)})Y(\Lambda^2_{(7)}S(x_{(2)}),\Lambda^1_{(5)}S(x_{(11)}),\Lambda^1_{(2)}S(x_{(14)}),\Lambda^2_{(8)}S(x_{(1)}),\Lambda^2_{(5)}S(x_{(4)}),\Lambda^1_{(7)}S(x_{(9)}),\\
&\Lambda^2_{(3)}S(x_{(6)}))\Lambda^1_{(9)})\\
&\lambda S(Z(\Lambda^1_{(3)}S(x_{(13)}),\Lambda^{1}_{(6)}S(x_{(10)}),\Lambda^2_{(4)}S(x_{(5)}),\Lambda^1_{(8)}S(x_{(8)}),\Lambda^2_{(2)}S(x_{(7)}),\Lambda^1_{(1)}S(x_{(15)}),\Lambda^1_{(4)}\\
&S(x_{(12)}),\Lambda^2_{(6)}S(x_{(3)}))\Lambda^2_{(9)})\\
\end{aligned}\]
$(2)$ \[\begin{aligned}
	&\lambda S (Y(\Lambda^2_{(1)},\Lambda^2_{(7)},\Lambda^1_{(5)},\Lambda^1_{(2)},\Lambda^2_{(8)},\Lambda^2_{(5)},\Lambda^1_{(7)})S^{-1}(\Lambda^2_{(3)})\Lambda^1_{(9)})\\
	&\lambda S(Y(\Lambda^1_{(3)},\Lambda^{1}_{(6)},\Lambda^2_{(4)},\Lambda^1_{(8)},\Lambda^2_{(2)},\Lambda^1_{(1)},\Lambda^1_{(4)},\Lambda^2_{(6)})S^{-1}(x)\Lambda^2_{(9)})\\
	=&\lambda S (Y(x_{(1)}\Lambda^2_{(1)},x_{(14)}\Lambda^2_{(7)},x_{(7)}\Lambda^1_{(5)},x_{(4)}\Lambda^1_{(2)},x_{(15)}\Lambda^2_{(8)},x_{(12)}\Lambda^2_{(5)},x_{(9)}\Lambda^1_{(7)}S^{-1}(\Lambda^2_{(3)})\Lambda^1_{(9)})\\
	&\lambda S(Z(x_{(5)}\Lambda^1_{(3)},x_{(8)}\Lambda^{1}_{(6)},x_{(11)}\Lambda^2_{(4)},x_{(10)}\Lambda^1_{(8)},x_{(2)}\Lambda^2_{(2)},x_{(3)}\Lambda^1_{(1)},x_{(6)}\Lambda^1_{(4)},x_{(13)}\Lambda^2_{(6)})\Lambda^2_{(9)})\\
\end{aligned}\]
\end{lemma}
\begin{proof}
Apply Theorem \ref{integral}(3) to $\lambda S(x\cdots \Lambda^2_{(9)})$.
 \[\begin{aligned}
	&\lambda S (S^{-1}(\Lambda^2_{(1)})Y(\Lambda^2_{(7)},\Lambda^1_{(5)},\Lambda^1_{(2)},\Lambda^2_{(8)},\Lambda^2_{(5)},\Lambda^1_{(7)})S^{-1}(\Lambda^2_{(3)})\Lambda^1_{(9)})\\
	&\lambda S(xZ(\Lambda^1_{(3)},\Lambda^{1}_{(6)},\Lambda^2_{(4)},\Lambda^1_{(8)},\Lambda^2_{(2)},\Lambda^1_{(1)},\Lambda^1_{(4)},\Lambda^2_{(6)})\Lambda^2_{(9)})\\
	=&\lambda S (x_{(8)}S^{-1}(\Lambda^2_{(1)})Y(\Lambda^2_{(7)}S(x_{(2)}),\Lambda^1_{(5)},\Lambda^1_{(2)},\Lambda^2_{(8)}S(x_{(1)}),\Lambda^2_{(5)}S(x_{(4)}),\Lambda^1_{(7)},\Lambda^2_{(3)}S(x_{(6)}))\Lambda^1_{(9)})\\
	&\lambda S(Z(\Lambda^1_{(3)},\Lambda^{1}_{(6)},\Lambda^2_{(4)}S(x_{(5)}),\Lambda^1_{(8)},\Lambda^2_{(2)}S(x_{(7)}),\Lambda^1_{(1)},\Lambda^1_{(4)},\Lambda^2_{(6)}S(x_{(3)}))\Lambda^2_{(9)})\\
\end{aligned}\]	
Continue to apply Theorem \ref{integral}(3) to $\lambda S(x_{(8)}\cdots \Lambda^1_{(9)})$, then we get $(1)$. Similarly we can obtain $(2)$ by using Theorem \ref{integral}(1) twice.
\end{proof}

Plug $S^{-1}(u)=f^{[2]}_rS^{-1}(f^{[1]}_r)$ into $\lambda S(S^{-1}(u)\cdots\Lambda^2_{(9)})$ then we obtain the following equality by Lemma \ref{lemma1}(1).
\[
\begin{aligned}
Z_F
=&\lambda S (S^{-1}(\Lambda^2_{(1)})d^{[15]}_qf^{[2]}_pf^{[2]}_{r(2)}S^{-1}(\Lambda^2_{(7)})d^{[2]}_qS^{-1}(u)S^{-2}(u^{-1})S^{-2}(f^{[11]}_p)S^{-2}(f^{[2]}_{r(11)})S^{-3}(\Lambda^1_{(5)})\\
&S^{-2}(d^{[9]}_q)S^{-2}(u)S^{-1}(u^{-1})f^{[14]}_pf^{[2]}_{r(14)}S^{-1}(\Lambda^1_{(2)})d^{[12]}_qf^{[1]}_pf^{[2]}_{r(1)}S^{-1}(\Lambda^2_{(8)})d^{[1]}_qS^{-1}(u)S^{-2}(u^{-1})\\
&S^{-2}(f^{[4]}_p)S^{-2}(f^{[2]}_{r(4)})S^{-3}(\Lambda^2_{(5)})S^{-2}(d^{[4]}_q)S^{-2}(u)S^{-1}(d^{[7]}_q)S^{-2}(\Lambda^1_{(7)})S^{-1}(f^{[2]}_{r(9)})S^{-1}(f^{[9]}_p)\\
&S^{-1}(u^{-1})f^{[6]}_pf^{[2]}_{r(6)}S^{-1}(\Lambda^2_{(3)})\Lambda^1_{(9)})\\
&\lambda S(S^{-1}(f^{[1]}_r)S^{-1}(d^{[11]}_q)S^{-2}(\Lambda^1_{(3)})S^{-1}(f^{[2]}_{r(13)})S^{-1}(f^{[13]}_p)S^{-2}(u^{-1})S^{-3}(u)S^{-3}(d^{[8]}_q)S^{-4}(\Lambda^{1}_{(6)})\\
&S^{-3}(f^{[2]}_{r(10)})S^{-3}(f^{[10]}_p)S^{-3}(u^{-1})S^{-2}(f^{[5]}_p)S^{-2}(f^{[2]}_{r(5)})S^{-3}(\Lambda^2_{(4)})S^{-2}(d^{[5]}_q)S^{-2}(u)S^{-1}(d^{[6]}_q)\\
&S^{-2}(\Lambda^1_{(8)})S^{-1}(f^{[2]}_{r(8)})S^{-1}(f^{[8]}_p)S^{-1}(u^{-1})f^{[7]}_pf^{[2]}_{r(7)}S^{-1}(\Lambda^2_{(2)})d^{[14]}_qS^{-1}(u)S^{-1}(d^{[13]}_q)S^{-2}(\Lambda^1_{(1)})\\
&S^{-1}(f^{[2]}_{r(15)})S^{-1}(f^{[15]}_p)S^{-2}(u^{-1})S^{-3}(u)S^{-3}(d^{[10]}_q)S^{-4}(\Lambda^1_{(4)})S^{-3}(f^{[2]}_{r(12)})S^{-3}(f^{[12]}_p)S^{-3}(u^{-1})\\
&S^{-2}(u)S^{-1}(d^{[3]}_q)S^{-2}(\Lambda^2_{(6)})S^{-1}(f^{[2]}_{r(3)})S^{-1}(f^{[3]}_p)S^{-1}(u^{-1})\Lambda^2_{(9)})\,.
\end{aligned}\]
Then we have 
\[\begin{aligned}
Z_F=&\lambda S (S^{-1}(\Lambda^2_{(1)})d^{[15]}_qf^{[3]}_pS^{-1}(\Lambda^2_{(7)})d^{[2]}_qS^{-1}(u)S^{-2}(u^{-1})S^{-2}(f^{[12]}_p)S^{-3}(\Lambda^1_{(5)})S^{-2}(d^{[9]}_q)S^{-2}(u)\\
&S^{-1}(u^{-1})f^{[15]}_pS^{-1}(\Lambda^1_{(2)})d^{[12]}_qf^{[2]}_pS^{-1}(\Lambda^2_{(8)})d^{[1]}_qS^{-1}(u)S^{-2}(u^{-1})S^{-2}(f^{[5]}_p)S^{-3}(\Lambda^2_{(5)})S^{-2}(d^{[4]}_q)\\
&S^{-2}(u)S^{-1}(d^{[7]}_q)S^{-2}(\Lambda^1_{(7)})S^{-1}(f^{[10]}_p)S^{-1}(u^{-1})f^{[7]}_pS^{-1}(\Lambda^2_{(3)})\Lambda^1_{(9)})\\
&\lambda S(S^{-1}(f^{[1]}_p)S^{-1}(d^{[11]}_q)S^{-2}(\Lambda^1_{(3)})S^{-1}(f^{[14]}_p)S^{-2}(u^{-1})S^{-3}(u)S^{-3}(d^{[8]}_q)S^{-4}(\Lambda^{1}_{(6)})S^{-3}(f^{[11]}_p)\\
&S^{-3}(u^{-1})S^{-2}(f^{[6]}_p)S^{-3}(\Lambda^2_{(4)})S^{-2}(d^{[5]}_q)S^{-2}(u)S^{-1}(d^{[6]}_q)S^{-2}(\Lambda^1_{(8)})S^{-1}(f^{[9]}_p)S^{-1}(u^{-1})f^{[8]}_p\\
&S^{-1}(\Lambda^2_{(2)})d^{[14]}_qS^{-1}(u)S^{-1}(d^{[13]}_q)S^{-2}(\Lambda^1_{(1)})S^{-1}(f^{[16]}_p)S^{-2}(u^{-1})S^{-3}(u)S^{-3}(d^{[10]}_q)S^{-4}(\Lambda^1_{(4)})\\
&S^{-3}(f^{[13]}_p)S^{-3}(u^{-1})S^{-2}(u)S^{-1}(d^{[3]}_q)S^{-2}(\Lambda^2_{(6)})S^{-1}(f^{[4]}_p)S^{-1}(u^{-1})\Lambda^2_{(9)})\\
=&\lambda S (S^{-1}(\Lambda^2_{(1)})d^{[9]}_qf^{[3]}_pS^{-1}(\Lambda^2_{(7)})d^{[2]}_qS^{-1}(u)S^{-2}(u^{-1})S^{-2}(f^{[6]}_p)S^{-3}(\Lambda^1_{(5)})S^{-2}(d^{[5]}_q)S^{-2}(u)\\
&S^{-1}(u^{-1})f^{[9]}_pS^{-1}(\Lambda^1_{(2)})d^{[8]}_qf^{[1]}_pS^{-1}(\Lambda^2_{(8)})d^{[1]}_qS^{-1}(u)S^{-2}(u^{-1})S^{-2}(f^{[5]}_p)S^{-3}(\Lambda^2_{(5)})S^{-2}(\Lambda^1_{(7)})\\
&S^{-1}(\Lambda^2_{(3)})\Lambda^1_{(9)})\\
&\lambda S(S^{-1}(f^{[1]}_p)S^{-1}(d^{[7]}_q)S^{-2}(\Lambda^1_{(3)})S^{-1}(f^{[8]}_p)S^{-2}(u^{-1})S^{-3}(u)S^{-3}(d^{[4]}_q)S^{-4}(\Lambda^{1}_{(6)})S^{-3}(\Lambda^2_{(4)})\\
&S^{-2}(\Lambda^1_{(8)})S^{-1}(\Lambda^2_{(2)})S^{-2}(\Lambda^1_{(1)})S^{-1}(f^{[10]}_p)S^{-2}(u^{-1})S^{-3}(u)S^{-3}(d^{[6]}_q)S^{-4}(\Lambda^1_{(4)})S^{-3}(f^{[7]}_p)\\
&S^{-3}(u^{-1})S^{-2}(u)S^{-1}(d^{[3]}_q)S^{-2}(\Lambda^2_{(6)})S^{-1}(f^{[4]}_p)S^{-1}(u^{-1})\Lambda^2_{(9)})\,.
\end{aligned}
\]
In the last equality of the above equations, Proposition \ref{cocycle}(6) is used. Now we apply Theorem \ref{integral}(1) twice to $S^{-1}(f^{[3]}_p)S^{-1}(u^{-1})\Lambda^2_{(9)}$ and obtain
\[
\begin{aligned}
	Z_F=&\lambda S (S^{-1}(\Lambda^2_{(1)})S^{-1}(u^{-1}_{(1)})d^{[9]}_qf^{[3]}_pS^{-1}(\Lambda^2_{(7)})S^{-1}(u^{-1}_{(14)})d^{[2]}_qS^{-1}(u)S^{-2}(u^{-1})S^{-2}(f^{[6]}_p)S^{-3}(\Lambda^1_{(5)})\\
	&S^{-3}(u^{-1}_{(7)})S^{-2}(d^{[5]}_q)S^{-2}(u)S^{-1}(u^{-1})f^{[9]}_pS^{-1}(\Lambda^1_{(2)})S^{-1}(u^{-1}_{(4)})d^{[8]}_qf^{[2]}_pS^{-1}(\Lambda^2_{(8)})S^{-1}(u^{-1}_{(15)})d^{[1]}_q\\
	&S^{-1}(u)S^{-2}(u^{-1})S^{-2}(f^{[5]}_p)S^{-3}(\Lambda^2_{(5)})S^{-3}(u^{-1}_{(12)})S^{-2}(u^{-1}_{(9)})S^{-2}(\Lambda^1_{(7)})S^{-1}(\Lambda^2_{(3)})\Lambda^1_{(9)})\\
	&\lambda S(S^{-1}(f^{[1]}_p)S^{-1}(d^{[7]}_q)S^{-2}(u^{-1}_{(5)})S^{-2}(\Lambda^1_{(3)})S^{-1}(f^{[8]}_p)S^{-2}(u^{-1})S^{-3}(u)S^{-3}(d^{[4]}_q)S^{-4}(u^{-1}_{(8)})\\
	&S^{-4}(\Lambda^{1}_{(6)})S^{-3}(\Lambda^2_{(4)})S^{-3}(u^{-1}_{(11)})S^{-2}(u^{-1}_{(10)})S^{-2}(\Lambda^1_{(8)})S^{-1}(\Lambda^2_{(2)})S^{-1}(u^{-1}_{(2)})S^{-2}(u^{-1}_{(3)})S^{-2}(\Lambda^1_{(1)})\\
	&S^{-1}(f^{[10]}_p)S^{-2}(u^{-1})S^{-3}(u)S^{-3}(d^{[6]}_q)S^{-4}(u^{-1}_{(6)})S^{-4}(\Lambda^1_{(4)})S^{-3}(f^{[7]}_p)S^{-3}(u^{-1})S^{-2}(u)S^{-1}(d^{[3]}_q)\\
	&S^{-2}(u^{-1}_{(13)})S^{-2}(\Lambda^2_{(6)})S^{-1}(f^{[4]}_p)\Lambda^2_{(9)})\\
	=&\lambda S (S^{-1}(\Lambda^2_{(1)})S^{-1}(u^{-1}_{(1)})d^{[9]}_qf^{[3]}_pS^{-1}(\Lambda^2_{(7)})S^{-1}(u^{-1}_{(8)})d^{[2]}_qS^{-1}(u)S^{-2}(u^{-1})S^{-2}(f^{[6]}_p)S^{-3}(\Lambda^1_{(5)})\\
	&S^{-3}(u^{-1}_{(5)})S^{-2}(d^{[5]}_q)S^{-2}(u)S^{-1}(u^{-1})f^{[9]}_pS^{-1}(\Lambda^1_{(2)})S^{-1}(u^{-1}_{(2)})d^{[8]}_qf^{[2]}_pS^{-1}(\Lambda^2_{(8)})S^{-1}(u^{-1}_{(9)})d^{[1]}_q\\
	&S^{-1}(u)S^{-2}(u^{-1})S^{-2}(f^{[5]}_p)S^{-3}(\Lambda^2_{(5)})S^{-2}(\Lambda^1_{(7)})S^{-1}(\Lambda^2_{(3)})\Lambda^1_{(9)})\\
	&\lambda S(S^{-1}(f^{[1]}_p)S^{-1}(d^{[7]}_q)S^{-2}(u^{-1}_{(3)})S^{-2}(\Lambda^1_{(3)})S^{-1}(f^{[8]}_p)S^{-2}(u^{-1})S^{-3}(u)S^{-3}(d^{[4]}_q)S^{-4}(u^{-1}_{(6)})\\
	&S^{-4}(\Lambda^{1}_{(6)})S^{-3}(\Lambda^2_{(4)})S^{-2}(\Lambda^1_{(8)})S^{-1}(\Lambda^2_{(2)})S^{-2}(\Lambda^1_{(1)})S^{-1}(f^{[10]}_p)S^{-2}(u^{-1})S^{-3}(u)S^{-3}(d^{[6]}_q)\\
	&S^{-4}(u^{-1}_{(4)})S^{-4}(\Lambda^1_{(4)})S^{-3}(f^{[7]}_p)S^{-3}(u^{-1})S^{-2}(u)S^{-1}(d^{[3]}_q)S^{-2}(u^{-1}_{(7)})S^{-2}(\Lambda^2_{(6)})S^{-1}(f^{[4]}_p)\Lambda^2_{(9)})
\end{aligned}\]
By Proposition \ref{cocycle}(8), we have that
\[\begin{aligned}
	Z_F=&\lambda S (S^{-1}(\Lambda^2_{(1)})S^{-1}(f^{[1]}_q)S^{-1}(u^{-1})f^{[3]}_pS^{-1}(\Lambda^2_{(7)})S^{-1}(f^{[8]}_q)S^{-2}(u^{-1})S^{-2}(f^{[6]}_p)S^{-3}(\Lambda^1_{(5)})\\
	&S^{-3}(f^{[5]}_q)S^{-3}(u^{-1})S^{-2}(u)S^{-1}(u^{-1})f^{[9]}_pS^{-1}(\Lambda^1_{(2)})S^{-1}(f^{[2]}_q)S^{-1}(u^{-1})f^{[2]}_pS^{-1}(\Lambda^2_{(8)})S^{-1}(f^{[9]}_q)\\
	&S^{-2}(u^{-1})S^{-2}(f^{[5]}_p)S^{-3}(\Lambda^2_{(5)})S^{-2}(\Lambda^1_{(7)})S^{-1}(\Lambda^2_{(3)})\Lambda^1_{(9)})\\
	&\lambda S(S^{-1}(f^{[1]}_p)S^{-2}(u^{-1})S^{-2}(f^{[3]}_q)S^{-2}(\Lambda^1_{(3)})S^{-1}(f^{[8]}_p)S^{-2}(u^{-1})S^{-3}(u)S^{-4}(u^{-1})S^{-4}(f^{[6]}_q)\\
	&S^{-4}(\Lambda^{1}_{(6)})S^{-3}(\Lambda^2_{(4)})S^{-2}(\Lambda^1_{(8)})S^{-1}(\Lambda^2_{(2)})S^{-2}(\Lambda^1_{(1)})S^{-1}(f^{[10]}_p)S^{-2}(u^{-1})S^{-3}(u)S^{-4}(u^{-1})\\
	&S^{-4}(f^{[4]}_q)S^{-4}(\Lambda^1_{(4)})S^{-3}(f^{[7]}_p)S^{-3}(u^{-1})S^{-2}(f^{[7]}_q)S^{-2}(\Lambda^2_{(6)})S^{-1}(f^{[4]}_p)\Lambda^2_{(9)})\\
\end{aligned}
\]
Then a similar calculation can be made for the term $S^{-1}(f^{[4]}_p)\Lambda^2_{(9)}$. 

\[
\begin{aligned}
Z_F=&\lambda S (S^{-1}(\Lambda^2_{(1)})S^{-1}(f^{[4]}_{p(1)})S^{-1}(f^{[1]}_q)S^{-1}(u^{-1})f^{[3]}_pS^{-1}(\Lambda^2_{(7)})S^{-1}(f^{[4]}_{p(8)})S^{-1}(f^{[8]}_q)S^{-2}(u^{-1})S^{-2}(f^{[6]}_p)\\
&S^{-3}(\Lambda^1_{(5)})S^{-3}(f^{[4]}_{p(5)})S^{-3}(f^{[5]}_q)S^{-3}(u^{-1})S^{-2}(u)S^{-1}(u^{-1})f^{[9]}_pS^{-1}(\Lambda^1_{(2)})S^{-1}(f^{[4]}_{p(2)})S^{-1}(f^{[2]}_q)\\
&S^{-1}(u^{-1})f^{[2]}_pS^{-1}(\Lambda^2_{(8)})S^{-1}(f^{[4]}_{p(9)})S^{-1}(f^{[9]}_q)S^{-2}(u^{-1})S^{-2}(f^{[5]}_p)S^{-3}(\Lambda^2_{(5)})S^{-2}(\Lambda^1_{(7)})S^{-1}(\Lambda^2_{(3)})\Lambda^1_{(9)})\\
&\lambda S(S^{-1}(f^{[1]}_p)S^{-2}(u^{-1})S^{-2}(f^{[3]}_q)S^{-2}(f^{[4]}_{p(3)})S^{-2}(\Lambda^1_{(3)})S^{-1}(f^{[8]}_p)S^{-2}(u^{-1})S^{-3}(u)S^{-4}(u^{-1})S^{-4}(f^{[6]}_q)\\
&S^{-4}(f^{[4]}_{p(6)})S^{-4}(\Lambda^{1}_{(6)})S^{-3}(\Lambda^2_{(4)})S^{-2}(\Lambda^1_{(8)})S^{-1}(\Lambda^2_{(2)})S^{-2}(\Lambda^1_{(1)})S^{-1}(f^{[10]}_p)S^{-2}(u^{-1})S^{-3}(u)S^{-4}(u^{-1})\\
&S^{-4}(f^{[4]}_q)S^{-4}(f^{[4]}_{p(4)})S^{-4}(\Lambda^1_{(4)})S^{-3}(f^{[7]}_p)S^{-3}(u^{-1})S^{-2}(f^{[7]}_q)S^{-2}(f^{[4]}_{p(7)})S^{-2}(\Lambda^2_{(6)})\Lambda^2_{(9)})\\
=&\lambda S (S^{-1}(\Lambda^2_{(1)})S^{-1}(f^{[4]}_p)S^{-1}(u^{-1})f^{[3]}_pS^{-1}(\Lambda^2_{(7)})S^{-1}(f^{[11]}_p)S^{-2}(u^{-1})S^{-2}(f^{[14]}_p)S^{-3}(\Lambda^1_{(5)})\\
&S^{-3}(f^{[8]}_p)S^{-3}(u^{-1})S^{-2}(u)S^{-1}(u^{-1})f^{[17]}_pS^{-1}(\Lambda^1_{(2)})S^{-1}(f^{[5]}_p)S^{-1}(u^{-1})f^{[2]}_pS^{-1}(\Lambda^2_{(8)})\\
&S^{-1}(f^{[12]}_p)S^{-2}(u^{-1})S^{-2}(f^{[13]}_p)S^{-3}(\Lambda^2_{(5)})S^{-2}(\Lambda^1_{(7)})S^{-1}(\Lambda^2_{(3)})\Lambda^1_{(9)})\\
&\lambda S(S^{-1}(f^{[1]}_p)S^{-2}(u^{-1})S^{-2}(f^{[6]}_p)S^{-2}(\Lambda^1_{(3)})S^{-1}(f^{[16]}_p)S^{-2}(u^{-1})S^{-3}(u)S^{-4}(u^{-1})S^{-4}(f^{[9]}_p)\\
&S^{-4}(\Lambda^{1}_{(6)})S^{-3}(\Lambda^2_{(4)})S^{-2}(\Lambda^1_{(8)})S^{-1}(\Lambda^2_{(2)})S^{-2}(\Lambda^1_{(1)})S^{-1}(f^{[18]}_p)S^{-2}(u^{-1})S^{-3}(u)S^{-4}(u^{-1})\\
&S^{-4}(f^{[7]}_p)S^{-4}(\Lambda^1_{(4)})S^{-3}(f^{[15]}_p)S^{-3}(u^{-1})S^{-2}(f^{[10]}_p)S^{-2}(\Lambda^2_{(6)})\Lambda^2_{(9)})\\
=&\lambda S (S^{-1}(\Lambda^2_{(1)})S^{-1}(\Lambda^2_{(7)})S^{-3}(\Lambda^1_{(5)})S^{-3}(f^{[2]}_r)S^{-3}(u^{-1})S^{-2}(u)S^{-1}(u^{-1})f^{[5]}_rS^{-1}(\Lambda^1_{(2)})S^{-1}(\Lambda^2_{(8)})\\
&S^{-3}(\Lambda^2_{(5)})S^{-2}(\Lambda^1_{(7)})S^{-1}(\Lambda^2_{(3)})\Lambda^1_{(9)})\\
&\lambda S(S^{-2}(\Lambda^1_{(3)})S^{-1}(f^{[4]}_r)S^{-2}(u^{-1})S^{-3}(u)S^{-4}(u^{-1})S^{-4}(f^{[3]}_r)S^{-4}(\Lambda^{1}_{(6)})S^{-3}(\Lambda^2_{(4)})S^{-2}(\Lambda^1_{(8)})\\
&S^{-1}(\Lambda^2_{(2)})S^{-2}(\Lambda^1_{(1)})S^{-1}(f^{[6]}_r)S^{-2}(u^{-1})S^{-3}(u)S^{-4}(u^{-1})S^{-4}(f^{[1]}_r)S^{-4}(\Lambda^1_{(4)})S^{-2}(\Lambda^2_{(6)})\Lambda^2_{(9)})\,.
\end{aligned}
\]

\begin{lemma}\label{lemma2}
	For any $x\in H$ and $X, Y\in \mathrm{End}(H)$,
\[\begin{aligned}
	\lambda S(X(\Lambda_{(1)}x)Y(\Lambda_{(2)})\Lambda_{(3)})=\lambda S(S^{-1}(x_{(1)})X(\Lambda_{(1)})Y(\Lambda_{(2)}S(x_{(2)}))\Lambda_{(3)}).
\end{aligned}\]
\end{lemma}
\begin{proof} By $\Lambda_{(1)}x\otimes \Lambda_{(2)}=\Lambda_{(1)}\otimes \Lambda_{(2)}S(x\leftharpoonup\alpha)=\Lambda_{(1)}\otimes \Lambda_{(2)}\alpha(x_{(1)})S(x_{(2)})$, the left hand side is equal to
\[\begin{aligned}
L=&\lambda (S(\Lambda_{(3)})SY(\Lambda_{(2)})SX(\Lambda_{(1)}x)))\\
=&\lambda (\alpha(x_{(1)})S(\Lambda_{(3)}S(x_{(2)}))SY(\Lambda_{(2)}S(x_{(3)}))SX(\Lambda_{(1)}))\\
=&\lambda (S^2(x_{(1)}\leftharpoonup\alpha)S(\Lambda_{(3)})SY(\Lambda_{(2)}S(x_{(2)}))SX(\Lambda_{(1)}))
\end{aligned}\]
By $\lambda(S^2(b\leftharpoonup\alpha)a)=\lambda(ab)$,
\[\begin{aligned}
	L=&\lambda (S(\Lambda_{(3)})SY(\Lambda_{(2)}S(x_{(2)}))SX(\Lambda_{(1)})x_{(1)})\\
	=&\lambda S(S^{-1}(x_{(1)})X(\Lambda_{(1)})Y(\Lambda_{(2)}S(x_{(2)}))\Lambda_{(3)})\,,
\end{aligned}\]
and this completes the proof.
\end{proof}

We are now ready to state and prove the main technical lemma in the paper.

\begin{lemma}\label{lem:3}
For any $x, w_1, w_2, w_3\in H$,
\[\begin{aligned}
	&\lambda S (S^{-1}(\Lambda^2_{(1)})S^{-1}(\Lambda^2_{(7)})S^{-3}(\Lambda^1_{(5)})w_2S^{-1}(\Lambda^1_{(2)}x_{(2)})S^{-1}(\Lambda^2_{(8)})S^{-3}(\Lambda^2_{(5)})S^{-2}(\Lambda^1_{(7)})S^{-1}(\Lambda^2_{(3)})\Lambda^1_{(9)})\\
	&\lambda S(S^{-2}(\Lambda^1_{(3)}x_{(3)})w_3S^{-4}(\Lambda^1_{(6)})S^{-3}(\Lambda^2_{(4)})S^{-2}(\Lambda^1_{(8)})S^{-1}(\Lambda^2_{(2)})S^{-2}(\Lambda^1_{(1)}x_{(1)})w_1S^{-4}(\Lambda^1_{(4)})S^{-2}(\Lambda^2_{(6)})\Lambda^2_{(9)})\\
	=&\lambda S (S^{-1}(\Lambda^2_{(1)})S^{-1}(\Lambda^2_{(7)})S^{-3}(x_{(2)}\Lambda^1_{(5)})w_2S^{-1}(\Lambda^1_{(2)})S^{-1}(\Lambda^2_{(8)})S^{-3}(\Lambda^2_{(5)})S^{-2}(\Lambda^1_{(7)})S^{-1}(\Lambda^2_{(3)})\Lambda^1_{(9)})\\
	&\lambda S(S^{-2}(\Lambda^1_{(3)})w_3S^{-4}(x_{(3)}\Lambda^1_{(6)})S^{-3}(\Lambda^2_{(4)})S^{-2}(\Lambda^1_{(8)})S^{-1}(\Lambda^2_{(2)})S^{-2}(\Lambda^1_{(1)})w_1S^{-4}(x_{(1)}\Lambda^1_{(4)})S^{-2}(\Lambda^2_{(6)})\Lambda^2_{(9)})\\
\end{aligned}\]
\end{lemma}

\begin{proof} 
Apply Lemma \ref{lemma2} for $\Lambda^1_{(1)}x_{(1)}\otimes\Lambda^1_{(2)}x_{(2)}\otimes\Lambda^1_{(3)}x_{(3)}$. The left hand side becomes
\[
\begin{aligned}
	L=&\lambda S (S^{-1}(x_{(1)})S^{-1}(\Lambda^2_{(1)})S^{-1}(\Lambda^2_{(7)})S^{-2}(x_{(5)})S^{-3}(\Lambda^1_{(5)})w_2S^{-1}(\Lambda^1_{(2)})S^{-1}(\Lambda^2_{(8)})S^{-3}(\Lambda^2_{(5)})S^{-2}(\Lambda^1_{(7)})\\
	&S^{-1}(x_{(3)})S^{-1}(\Lambda^2_{(3)})\Lambda^1_{(9)})\\
	&\lambda S(S^{-2}(\Lambda^1_{(3)})w_3S^{-4}(\Lambda^1_{(6)})S^{-2}(x_{(4)})S^{-3}(\Lambda^2_{(4)})S^{-2}(\Lambda^1_{(8)})S^{-1}(x_{(2)})S^{-1}(\Lambda^2_{(2)})S^{-2}(\Lambda^1_{(1)})w_1S^{-4}(\Lambda^1_{(4)})\\
	&S^{-3}(x_{(6)})S^{-2}(\Lambda^2_{(6)})\Lambda^2_{(9)})\\
\end{aligned}
\]
Similarly, Lemma \ref{lemma2} is applied to $S^{-1}(x_{(1)})S^{-1}(\Lambda^2_{(1)})=S^{-1}(\Lambda^2_{(1)}x_{(1)})$. Then we get that
\[
\begin{aligned}
L=&\lambda S (S^{-1}(\Lambda^2_{(1)})x_{(3)}S^{-1}(\Lambda^2_{(7)})S^{-2}(x_{(12)})S^{-3}(\Lambda^1_{(5)})w_2S^{-1}(\Lambda^1_{(2)})x_{(2)}S^{-1}(\Lambda^2_{(8)})S^{-2}(x_{(5)})S^{-3}(\Lambda^2_{(5)})\\
&S^{-2}(\Lambda^1_{(7)})S^{-1}(x_{(10)})x_{(7)}S^{-1}(\Lambda^2_{(3)})\Lambda^1_{(9)})\\
&\lambda S(S^{-1}(x_{(1)})S^{-2}(\Lambda^1_{(3)})w_3S^{-4}(\Lambda^1_{(6)})S^{-3}(x_{(11)})S^{-2}(x_{(6)})S^{-3}(\Lambda^2_{(4)})S^{-2}(\Lambda^1_{(8)})S^{-1}(x_{(9)})x_{(8)}\\
&S^{-1}(\Lambda^2_{(2)})S^{-2}(\Lambda^1_{(1)})w_1S^{-4}(\Lambda^1_{(4)})S^{-3}(x_{(13)})S^{-2}(\Lambda^2_{(6)})S^{-1}(x_{(4)})\Lambda^2_{(9)})\\
=&\lambda S (S^{-1}(\Lambda^2_{(1)})x_{(3)}S^{-1}(\Lambda^2_{(7)})S^{-2}(x_{(6)})S^{-3}(\Lambda^1_{(5)})w_2S^{-1}(\Lambda^1_{(2)})x_{(2)}S^{-1}(\Lambda^2_{(8)})S^{-2}(x_{(5)})S^{-3}(\Lambda^2_{(5)})\\
&S^{-2}(\Lambda^1_{(7)})S^{-1}(\Lambda^2_{(3)})\Lambda^1_{(9)})\\
&\lambda S(S^{-1}(x_{(1)})S^{-2}(\Lambda^1_{(3)})w_3S^{-4}(\Lambda^1_{(6)})S^{-3}(\Lambda^2_{(4)})S^{-2}(\Lambda^1_{(8)})S^{-1}(\Lambda^2_{(2)})S^{-2}(\Lambda^1_{(1)})w_1S^{-4}(\Lambda^1_{(4)})S^{-3}(x_{(7)})\\
&S^{-2}(\Lambda^2_{(6)})S^{-1}(x_{(4)})\Lambda^2_{(9)})\\
\end{aligned}
\]
By Lemma \ref{lemma1}(2), we get that
\[
\begin{aligned}
L=&\lambda S (S^{-1}(\Lambda^2_{(1)})S^{-1}(x_{(4)})x_{(3)}S^{-1}(\Lambda^2_{(7)})S^{-1}(x_{(17)})S^{-2}(x_{(20)})S^{-3}(\Lambda^1_{(5)})S^{-3}(x_{(10)})w_2S^{-1}(\Lambda^1_{(2)})\\
&S^{-1}(x_{(7)})x_{(2)}S^{-1}(\Lambda^2_{(8)})S^{-1}(x_{(18)})S^{-2}(x_{(19)})S^{-3}(\Lambda^2_{(5)})S^{-3}(x_{(15)})S^{-2}(x_{(12)})S^{-2}(\Lambda^1_{(7)})S^{-1}(\Lambda^2_{(3)})\Lambda^1_{(9)})\\
&\lambda S(S^{-1}(x_{(1)})S^{-2}(x_{(8)})S^{-2}(\Lambda^1_{(3)})w_3S^{-4}(x_{(11)})S^{-4}(\Lambda^1_{(6)})S^{-3}(\Lambda^2_{(4)})S^{-3}(x_{(14)})S^{-2}(x_{(13)})S^{-2}(\Lambda^1_{(8)})\\
&S^{-1}(\Lambda^2_{(2)})S^{-1}(x_{(5)})S^{-2}(x_{(6)})S^{-2}(\Lambda^1_{(1)})w_1S^{-4}(x_{(9)})S^{-4}(\Lambda^1_{(4)})S^{-3}(x_{(21)})S^{-2}(x_{(16)})S^{-2}(\Lambda^2_{(6)})\Lambda^2_{(9)})\\
=&\lambda S (S^{-1}(\Lambda^2_{(1)})S^{-1}(\Lambda^2_{(7)})S^{-3}(\Lambda^1_{(5)})S^{-3}(x_{(2)})w_2S^{-1}(\Lambda^1_{(2)})S^{-1}(\Lambda^2_{(8)})S^{-3}(\Lambda^2_{(5)})S^{-2}(\Lambda^1_{(7)})S^{-1}(\Lambda^2_{(3)})\Lambda^1_{(9)})\\
&\lambda S(S^{-2}(\Lambda^1_{(3)})w_3S^{-4}(x_{(3)})S^{-4}(\Lambda^1_{(6)})S^{-3}(\Lambda^2_{(4)})S^{-2}(\Lambda^1_{(8)})S^{-1}(\Lambda^2_{(2)})S^{-2}(\Lambda^1_{(1)})w_1S^{-4}(x_{(1)})S^{-4}(\Lambda^1_{(4)})\\
&S^{-2}(\Lambda^2_{(6)})\Lambda^2_{(9)})
\end{aligned}\]
This completes the proof.
\end{proof}

% With the help of Lemma \ref{lem:3}, we can obtain the gauge invariance of $Z(W, f_0, H)$.

\begin{prop}
Let $H$ be a finite-dimensional Hopf algebra. Then for the framing $f_0$ shown in Figure \ref{framed_weeks}, the corresponding Kuperberg invariant $Z(W, f_0, H)$ is a gauge invariant of $H$.
\end{prop}
\begin{proof}
As is noted above, it suffices to show that $Z(H) = Z_F$, where $Z(H) = \alpha(g)Z(W, f_0, H)$ (see Equation \eqref{eq:ZH}). Since $Q=uS(u^{-1})$, by the computations above and Lemma \ref{lem:3}, we have 
\[
\begin{aligned}
	Z_F
	=&\lambda S (S^{-1}(\Lambda^2_{(1)})S^{-1}(\Lambda^2_{(7)})S^{-3}(\Lambda^1_{(5)})S^{-3}(f^{[2]}_qf^{[1]}_{i(2)})S^{-3}(u^{-1})S^{-2}(Q)f^{[2]}_pS^{-1}(\Lambda^1_{(2)}S(f^{[2]}_i)_{(2)})\\
	&S^{-1}(\Lambda^2_{(8)})S^{-3}(\Lambda^2_{(5)})S^{-2}(\Lambda^1_{(7)})S^{-1}(\Lambda^2_{(3)})\Lambda^1_{(9)})\\
	&\lambda S(S^{-2}(\Lambda^1_{(3)}S(f^{[2]}_i)_{(3)})S^{-1}(f^{[1]}_p))S^{-3}(Q)S^{-4}(u^{-1})S^{-4}(f^{[3]}_qf^{[1]}_i)_{(3)})S^{-4}(\Lambda^{1}_{(6)})S^{-3}(\Lambda^2_{(4)})\\
	&S^{-2}(\Lambda^1_{(8)})S^{-1}(\Lambda^2_{(2)})S^{-2}(\Lambda^1_{(1)}S(f^{[2]}_i)_{(1)})S^{-1}(f^{[3]}_p)S^{-3}(Q)S^{-4}(u^{-1})S^{-4}(f^{[1]}_qf^{[1]}_{i(1)})S^{-4}(\Lambda^1_{(4)})\\
	&S^{-2}(\Lambda^2_{(6)})\Lambda^2_{(9)})\\
	=&\lambda S (S^{-1}(\Lambda^2_{(1)})S^{-1}(\Lambda^2_{(7)})S^{-3}(\Lambda^1_{(5)})S^{-3}(S(f^{[2]}_i)_{(2)})S^{-3}(f^{[2]}_qf^{[1]}_{i(2)})S^{-3}(u^{-1})S^{-2}(Q)f^{[2]}_p\\
	&S^{-1}(\Lambda^1_{(2)})S^{-1}(\Lambda^2_{(8)})S^{-3}(\Lambda^2_{(5)})S^{-2}(\Lambda^1_{(7)})S^{-1}(\Lambda^2_{(3)})\Lambda^1_{(9)})\\
	&\lambda S(S^{-2}(\Lambda^1_{(3)})S^{-1}(f^{[1]}_p)S^{-3}(Q)S^{-4}(u^{-1})S^{-4}(f^{[3]}_qf^{[1]}_i)_{(3)})S^{-4}(S(f^{[2]}_i)_{(3)})S^{-4}(\Lambda^{1}_{(6)})\\
	&S^{-3}(\Lambda^2_{(4)})S^{-2}(\Lambda^1_{(8)})S^{-1}(\Lambda^2_{(2)})S^{-2}(\Lambda^1_{(1)})S^{-1}(f^{[3]}_p)S^{-3}(Q)S^{-4}(u^{-1})S^{-4}(f^{[1]}_qf^{[1]}_{i(1)})\\
	&S^{-4}(S(f^{[2]}_i)_{(1)})S^{-4}(\Lambda^1_{(4)})S^{-2}(\Lambda^2_{(6)})\Lambda^2_{(9)})\,.
\end{aligned}\]
Then we can simplify $Z_F$ as
\[\begin{aligned}
	Z_F=&\lambda S (S^{-1}(\Lambda^2_{(1)})S^{-1}(\Lambda^2_{(7)})S^{-3}(\Lambda^1_{(5)})S^{-3}(f^{[2]}_qu_{(2)})S^{-3}(u^{-1})S^{-2}(Q)f^{[2]}_pS^{-1}(\Lambda^1_{(2)})S^{-1}(\Lambda^2_{(8)})\\
	&S^{-3}(\Lambda^2_{(5)})S^{-2}(\Lambda^1_{(7)})S^{-1}(\Lambda^2_{(3)})\Lambda^1_{(9)})\\
	&\lambda S(S^{-2}(\Lambda^1_{(3)})S^{-1}(f^{[1]}_p)S^{-3}(Q)S^{-4}(u^{-1})S^{-4}(f^{[3]}_qu_{(3)})S^{-4}(\Lambda^{1}_{(6)})S^{-3}(\Lambda^2_{(4)})S^{-2}(\Lambda^1_{(8)})\\
	&S^{-1}(\Lambda^2_{(2)})S^{-2}(\Lambda^1_{(1)})S^{-1}(f^{[3]}_p)S^{-3}(Q)S^{-4}(u^{-1})S^{-4}(f^{[1]}_qu_{(1)})S^{-4}(\Lambda^1_{(4)})S^{-2}(\Lambda^2_{(6)})\Lambda^2_{(9)})\\
	=&\lambda S (S^{-1}(\Lambda^2_{(1)})S^{-1}(\Lambda^2_{(7)})S^{-3}(\Lambda^1_{(5)})S^{-3}(uS(d^{[2]}_q))S^{-3}(u^{-1})S^{-2}(Q)f^{[2]}_pS^{-1}(\Lambda^1_{(2)})S^{-1}(\Lambda^2_{(8)})\\
	&S^{-3}(\Lambda^2_{(5)})S^{-2}(\Lambda^1_{(7)})S^{-1}(\Lambda^2_{(3)})\Lambda^1_{(9)})\\
	&\lambda S(S^{-2}(\Lambda^1_{(3)})S^{-1}(f^{[1]}_p)S^{-3}(Q)S^{-4}(u^{-1})S^{-4}(uS(d^{[1]}))S^{-4}(\Lambda^{1}_{(6)})S^{-3}(\Lambda^2_{(4)})S^{-2}(\Lambda^1_{(8)})\\
	&S^{-1}(\Lambda^2_{(2)})S^{-2}(\Lambda^1_{(1)})S^{-1}(f^{[3]}_p)S^{-3}(Q)S^{-4}(u^{-1})S^{-4}(uS(d^{[3]}))S^{-4}(\Lambda^1_{(4)})S^{-2}(\Lambda^2_{(6)})\Lambda^2_{(9)})\\
	=&\lambda S (S^{-1}(\Lambda^2_{(1)})S^{-1}(\Lambda^2_{(7)})S^{-3}(\Lambda^1_{(5)})S^{-2}(d^{[2]}_q)S^{-2}(Q)f^{[2]}_pS^{-1}(\Lambda^1_{(2)})S^{-1}(\Lambda^2_{(8)})S^{-3}(\Lambda^2_{(5)})\\
	&S^{-2}(\Lambda^1_{(7)})S^{-1}(\Lambda^2_{(3)})\Lambda^1_{(9)})\\
	&\lambda S(S^{-2}(\Lambda^1_{(3)})S^{-1}(f^{[1]}_p)S^{-3}(Q)S^{-3}(d^{[1]}_q)S^{-4}(\Lambda^{1}_{(6)})S^{-3}(\Lambda^2_{(4)})S^{-2}(\Lambda^1_{(8)})S^{-1}(\Lambda^2_{(2)})\\
	&S^{-2}(\Lambda^1_{(1)})S^{-1}(f^{[3]}_p)S^{-3}(Q)S^{-3}(d^{[3]}_q)S^{-4}(\Lambda^1_{(4)})S^{-2}(\Lambda^2_{(6)})\Lambda^2_{(9)})
\end{aligned}
\]

 By Proposition \ref{cocycle}(9), we have that

\[
\begin{aligned}
Z_F=&\lambda S (S^{-1}(\Lambda^2_{(1)})S^{-1}(\Lambda^2_{(7)})S^{-3}(\Lambda^1_{(5)})S^{-2}(Q_{(2)})S^{-1}(\Lambda^1_{(2)})S^{-1}(\Lambda^2_{(8)})S^{-3}(\Lambda^2_{(5)})S^{-2}(\Lambda^1_{(7)})\\
&S^{-1}(\Lambda^2_{(3)})\Lambda^1_{(9)})\\
&\lambda S(S^{-2}(\Lambda^1_{(3)})S^{-3}(Q_{(1)})S^{-4}(\Lambda^{1}_{(6)})S^{-3}(\Lambda^2_{(4)})S^{-2}(\Lambda^1_{(8)})S^{-1}(\Lambda^2_{(2)})S^{-2}(\Lambda^1_{(1)})S^{-3}(Q_{(3)})\\
&S^{-4}(\Lambda^1_{(4)})S^{-2}(\Lambda^2_{(6)})\Lambda^2_{(9)})
\end{aligned}\]
\[\begin{aligned}
=&\lambda S (S^{-1}(\Lambda^2_{(1)})S^{-1}(\Lambda^2_{(7)})S^{-3}(\Lambda^1_{(5)})S^{-2}(u_{(2)})S^{-1}(u^{-1}_{(2)})S^{-1}(\Lambda^1_{(2)})S^{-1}(\Lambda^2_{(8)})S^{-3}(\Lambda^2_{(5)})\\
&S^{-2}(\Lambda^1_{(7)})S^{-1}(\Lambda^2_{(3)})\Lambda^1_{(9)})\\
&\lambda S(S^{-2}(\Lambda^1_{(3)})S^{-2}(u^{-1}_{(3)})S^{-3}(u_{(1)})S^{-4}(\Lambda^{1}_{(6)})S^{-3}(\Lambda^2_{(4)})S^{-2}(\Lambda^1_{(8)})S^{-1}(\Lambda^2_{(2)})S^{-2}(\Lambda^1_{(1)})\\
&S^{-2}(u^{-1}_{(1)})S^{-3}(u_{(3)})S^{-4}(\Lambda^1_{(4)})S^{-2}(\Lambda^2_{(6)})\Lambda^2_{(9)})\\	
=&\lambda S (S^{-1}(\Lambda^2_{(1)})S^{-1}(\Lambda^2_{(7)})S^{-3}(\Lambda^1_{(5)})S^{-2}(u_{(2)})S^{-1}(d^{[2]}_{j(2)})d^{[1]}_{j(2)}S^{-1}(\Lambda^1_{(2)})S^{-1}(\Lambda^2_{(8)})S^{-3}(\Lambda^2_{(5)})\\
&S^{-2}(\Lambda^1_{(7)})S^{-1}(\Lambda^2_{(3)})\Lambda^1_{(9)})\\
&\lambda S(S^{-2}(\Lambda^1_{(3)})S^{-1}(d^{[1]}_{j(1)})S^{-2}(d^{[2]}_{j(3)})S^{-3}(u_{(1)})S^{-4}(\Lambda^{1}_{(6)})S^{-3}(\Lambda^2_{(4)})S^{-2}(\Lambda^1_{(8)})S^{-1}(\Lambda^2_{(2)})\\
&S^{-2}(\Lambda^1_{(1)})S^{-1}(d^{[1]}_{j(3)})S^{-2}(d^{[2]}_{j(1)})S^{-3}(u_{(3)})S^{-4}(\Lambda^1_{(4)})S^{-2}(\Lambda^2_{(6)})\Lambda^2_{(9)})\\
\end{aligned}\]
By Lemma \ref{lem:3}, we have 
\[\begin{aligned}
Z_F=&\lambda S (S^{-1}(\Lambda^2_{(1)})S^{-1}(\Lambda^2_{(7)})S^{-3}(\Lambda^1_{(5)})S^{-2}(d^{[1]}_{j(2)})S^{-2}(u_{(2)})S^{-1}(d^{[2]}_{j(2)})S^{-1}(\Lambda^1_{(2)})S^{-1}(\Lambda^2_{(8)})\\
&S^{-3}(\Lambda^2_{(5)})S^{-2}(\Lambda^1_{(7)})S^{-1}(\Lambda^2_{(3)})\Lambda^1_{(9)})\\
&\lambda S(S^{-2}(\Lambda^1_{(3)})S^{-2}(d^{[2]}_{j(3)})S^{-3}(u_{(1)})S^{-3}(d^{[1]}_{j(1)})S^{-4}(\Lambda^{1}_{(6)})S^{-3}(\Lambda^2_{(4)})S^{-2}(\Lambda^1_{(8)})S^{-1}(\Lambda^2_{(2)})\\
&S^{-2}(\Lambda^1_{(1)})S^{-2}(d^{[2]}_{j(1)})S^{-3}(u_{(3)})S^{-3}(d^{[1]}_{j(3)})S^{-4}(\Lambda^1_{(4)})S^{-2}(\Lambda^2_{(6)})\Lambda^2_{(9)})\\
=&\lambda S (S^{-1}(\Lambda^2_{(1)})S^{-1}(\Lambda^2_{(7)})S^{-3}(\Lambda^1_{(5)})S^{-2}(d^{[1]}_{j(2)}u_{(2)}S(d^{[2]}_j)_{(2)})S^{-1}(\Lambda^1_{(2)})S^{-1}(\Lambda^2_{(8)})S^{-3}(\Lambda^2_{(5)})\\
&S^{-2}(\Lambda^1_{(7)})S^{-1}(\Lambda^2_{(3)})\Lambda^1_{(9)})\\
&\lambda S(S^{-2}(\Lambda^1_{(3)})S^{-3}(d^{[1]}_{j(1)}u_{(1)}S(d^{[2]}_j)_{(1)})S^{-4}(\Lambda^{1}_{(6)})S^{-3}(\Lambda^2_{(4)})S^{-2}(\Lambda^1_{(8)})S^{-1}(\Lambda^2_{(2)})S^{-2}(\Lambda^1_{(1)})\\
&S^{-3}(d^{[1]}_{j(3)}u_{(3)}S(d^{[2]}_j)_{(3)})S^{-4}(\Lambda^1_{(4)})S^{-2}(\Lambda^2_{(6)})\Lambda^2_{(9)})\\
=&\lambda S (S^{-1}(\Lambda^2_{(1)})S^{-1}(\Lambda^2_{(7)})S^{-3}(\Lambda^1_{(5)})S^{-1}(\Lambda^1_{(2)})S^{-1}(\Lambda^2_{(8)})S^{-3}(\Lambda^2_{(5)})S^{-2}(\Lambda^1_{(7)})S^{-1}(\Lambda^2_{(3)})\Lambda^1_{(9)})\\
&\lambda S(S^{-2}(\Lambda^1_{(3)})S^{-4}(\Lambda^{1}_{(6)})S^{-3}(\Lambda^2_{(4)})S^{-2}(\Lambda^1_{(8)})S^{-1}(\Lambda^2_{(2)})S^{-2}(\Lambda^1_{(1)})S^{-4}(\Lambda^1_{(4)})S^{-2}(\Lambda^2_{(6)})\Lambda^2_{(9)})\\
=&Z(H)\,.
\end{aligned}
\]
This completes the proof.
\end{proof}

In general, for any framing $f$ on the Weeks manifold $W$, its Kuperberg invariant $Z(W, f, H)$ is related to $Z(W, f_0, H)$ in the following sense. By Proposition 2.2 in \cite{Kup96}, the framings of the Weeks manifold $W$ are determined by their spin structures in $H^1(W, \mathbb{Z}/2\mathbb{Z})$ and degrees in $H^3(W, \mathbb{Z})=\mathbb{Z}$. The presentation of the fundamental group of $W$,
$\pi_1(W)\cong\langle a, b~|~a^2b^2a^{-1}ba^{-1}b^2 = ab^2a^2b^{-1}ab^{-1}a = 1\rangle$ \cite{CK20}, gives $H_1(W, \mathbb{Z})=0$. By the universal coefficient theorem, we have $H^1(W, \mathbb{Z}/2\mathbb{Z})=0$. This means that the framings of $W$ differ in their degrees. When the degree changes by $+1$, the invariant $Z(W, f, H)$ gains a factor of $\alpha(g)$ \cite{Kup96}. As $\alpha(g)$ is a gauge invariant \cite{Shi15}, we can conclude that for any framing $f$ on the Weeks manifold, the Kuperberg invariant $Z(W, f, H)$ is a gauge invariant. In summary, we have obtained the following main theorem of this paper.

\begin{theorem}\label{weeksthm}
Let $H$ a finite-dimensional Hopf algebra. For any framing f on the Weeks manifold $W$, the Kuperberg invariant $Z(W, f, H)$ is a gauge
invariant of H. \qed
\end{theorem}

\section{3-torus}

The other gauge invariant we will study is the Kuperberg invariant of 3-torus $T^3$. It is the closed 3-manifold obtained by gluing opposite faces of a cube. It means that the Heegaard surface of $T^3$ is a genus three closed surface obtained from a 6-punctured 2-sphere by attaching three handles \cite{GS99}. This is the minimal genus Heegaard splitting of $T^3$.   

\begin{figure}\centering
\begin{tikzpicture}[scale=0.4, decoration={markings,
	mark=between positions 0 and 0.9 step 10mm with {\draw[->] (0.1, 0)--(0.2, 0);}}]
\draw [line width=0.3mm] (0, 12) circle(1);
\draw [line width=0.3mm] (13, 12) circle(1);
\draw [line width=0.3mm] (0, 0) circle(1);
\draw [line width=0.3mm] (13, 0) circle(1);
\draw [line width=0.3mm] (0, -12) circle(1);
\draw [line width=0.3mm] (13, -12) circle(1);
\draw [red, line width=0.2mm] (-1, 12) to [out=180, in=-90] (-2, 14) to [out=90, in=180] (13, 19) to [out=0, in=90] (20, 0) to [out=-90, in=0] (13, -16) to [out=180, in=-90] (8, -12) to [out=90, in=-90] (16, 0) to [out=90, in=0] (13, 3) to [out=180, in=90] (9.5, 0) to [out=-90, in=90] (2.5, -12) to [out=-90, in=0] (0, -15) to [out=180, in=-90] (-4, -12) to [out=90, in=-90] (6.5, 0) to [out=90, in=-45] (14, 12);
\draw [blue, line width=0.2mm] (-1, 0) to [out=150, in=-90] (8, 12) to [out=90, in=180] (12, 16) to [out=0, in=90] (17, 0) to [out=-90, in=90] (10, -12) to [out=-90, in=180] (13, -15) to [out=0, in=-90] (19, 0) to [out=90, in=0] (12.5, 18) to [out=180, in=90] (5, 12) to [out=-90, in=90] (-5, 0) to [out=-90, in=180] (0, -17) to [out=0, in=-90] (4, -12) to [out=90, in=-45] (14, 0);
\draw [black, line width=0.2mm] (-1, -12) to [out=150, in=-90] (8, 0) to [out=90, in=-90] (16, 11.5) to [out=90, in=0] (12, 15) to [out=180, in=90] (9.5, 12) to [out=-90, in=90] (5, 0) to [out=-90, in=0] (0, -3) to [out=180, in=-90] (-4, 0) to [out=90, in=-90] (6.5, 12) to [out=90, in=180] (12.5, 17) to [out=0, in=90] (18, 0) to [out=-90, in=30] (14, -12);
\draw [teal, line width=0.2mm] (1, 12) to (4, 12) to [out=0, in=180] (5, 11.5)  to [out=0, in=180] (6, 12) to (12, 12);
\draw [violet, line width=0.2mm] (1, 0) to (4, 0) to [out=0, in=180] (5, -0.5)  to [out=0, in=180] (6, 0) to (12, 0);
\draw [orange, line width=0.2mm] (1, -12) to (7, -12) to [out=0, in=180] (8, -12.5)  to [out=0, in=180] (9, -12) to (12, -12);
\draw [dashed, line width=0.2mm, postaction={decorate}] (6, -12) to (0, -12);
\draw [dashed, line width=0.2mm, postaction={decorate}] (6, -12) to (8, -12);
\draw [dashed, line width=0.2mm, postaction={decorate}] (13, -12) to (8, -12);
\draw [dashed, line width=0.2mm, postaction={decorate}] (3, 0) to (0, 0);
\draw [dashed, line width=0.2mm, postaction={decorate}] (3, 0) to (5, 0);
\draw [dashed, line width=0.2mm, postaction={decorate}] (13, 0) to (5, 0);
\draw [dashed, line width=0.2mm, postaction={decorate}] (-7, 7.5-12) to [out=0, in=-90] (3, 0);	
\draw [dashed, line width=0.2mm, postaction={decorate}] (-7, 0) to (0, 0);	
\draw [dashed, line width=0.2mm, postaction={decorate}] (-7, -3) to (-2, -3) to [out=0, in=-90] (1, -1) to [out=90, in=-30] (0, 0);	
\draw [dashed, line width=0.2mm, postaction={decorate}] (5, 0) to [out=-90, in=180] (21, 7.5-12);	
\draw [dashed, line width=0.2mm, postaction={decorate}] (13, 0) to (21, 0);	
\draw [dashed, line width=0.2mm, postaction={decorate}] (13, 0) to [out=-150, in=90] (11, -2) to [out=-90, in=180] (13, -3) to (21, -3);	
\draw [dashed, line width=0.2mm, postaction={decorate}] (-7, -6) to (21, -6);	
\draw [dashed, line width=0.2mm, postaction={decorate}] (6, -18) to (6, -12);	
\draw [dashed, line width=0.2mm, postaction={decorate}] (-7, 4.5-12) to [out=0, in=90] (6, -12);	
\draw [dashed, line width=0.2mm, postaction={decorate}] (5, -18) to [out=90, in=-30] (0, -12);	
\draw [dashed, line width=0.2mm, postaction={decorate}] (-7, -18) to (0, -12);	
\draw [dashed, line width=0.2mm, postaction={decorate}] (-7, -12) to (0, -12);	
\draw [dashed, line width=0.2mm, postaction={decorate}] (-7, -9) to (-2, -9) to [out=0, in=90] (1, -11) to [out=-90, in=30] (0, -12);	
\draw [dashed, line width=0.2mm, postaction={decorate}] (8, -12) to (8, -18);	
\draw [dashed, line width=0.2mm, postaction={decorate}] (8, -12) to [out=90, in=180] (21, 4.5-12);	
\draw [dashed, line width=0.2mm, postaction={decorate}] (13, -12) to [out=-150, in=90] (10, -18);	
\draw [dashed, line width=0.2mm, postaction={decorate}]  (13, -12) to (21, -18);	
\draw [dashed, line width=0.2mm, postaction={decorate}] (13, -12) to (21, -12);	
\draw [dashed, line width=0.2mm, postaction={decorate}] (13, -12) to [out=120, in=-90] (11, -10) to [out=90, in=180] (13, -9) to (21, -9);	
\draw [dashed, line width=0.2mm, postaction={decorate}] (0.5, -5) to [out=0, in=180] (4, -4) to [out=0, in=180] (7.5, -5);
\draw [dashed, line width=0.2mm, postaction={decorate}] (3.5, -7) to [out=0, in=180] (7, -8) to [out=0, in=180] (10.5, -7);
\draw [dashed, line width=0.2mm, postaction={decorate}] (6.3, -15) to [out=90, in=180] (7, -14) to [out=0, in=90] (7.7, -15);
\draw [dashed, line width=0.2mm, postaction={decorate}] (3, 12) to (0, 12);
\draw [dashed, line width=0.2mm, postaction={decorate}] (3, 12) to (5, 12);
\draw [dashed, line width=0.2mm, postaction={decorate}] (13, 12) to (5, 12);
\draw [dashed, line width=0.2mm, postaction={decorate}] (3, 19) to (3, 12);	
\draw [dashed, line width=0.2mm, postaction={decorate}] (-7, 7.5) to [out=0, in=-90] (3, 12);	
\draw [dashed, line width=0.2mm, postaction={decorate}] (2, 19) to [out=-90, in=60] (0, 12);	
\draw [dashed, line width=0.2mm, postaction={decorate}] (-7, 19) to (0, 12);	
\draw [dashed, line width=0.2mm, postaction={decorate}] (-7, 12) to (0, 12);	
\draw [dashed, line width=0.2mm, postaction={decorate}] (-7, 9) to (-2, 9) to [out=0, in=-90] (1, 11) to [out=90, in=-30] (0, 12);	
\draw [dashed, line width=0.2mm, postaction={decorate}] (5, 12) to (5, 19);	
\draw [dashed, line width=0.2mm, postaction={decorate}] (5, 12) to [out=-90, in=180] (21, 7.5);	
\draw [dashed, line width=0.2mm, postaction={decorate}] (13, 12) to [out=150, in=-90] (6, 19);	
\draw [dashed, line width=0.2mm, postaction={decorate}] (13, 12) to (21, 19);	
\draw [dashed, line width=0.2mm, postaction={decorate}] (13, 12) to (21, 12);	
\draw [dashed, line width=0.2mm, postaction={decorate}] (13, 12) to [out=-150, in=90] (11, 10) to [out=-90, in=180] (13, 9) to (21, 9);	
\draw [dashed, line width=0.2mm, postaction={decorate}] (-7, 6) to (21, 6);	
\draw [dashed, line width=0.2mm, postaction={decorate}] (-7, 4.5) to [out=0, in=90] (3, 0);	
\draw [dashed, line width=0.2mm, postaction={decorate}] (-7, 0) to (0, 0);	
\draw [dashed, line width=0.2mm, postaction={decorate}] (-7, 3) to (-2, 3) to [out=0, in=90] (1, 1) to [out=-90, in=30] (0, 0);	
\draw [dashed, line width=0.2mm, postaction={decorate}] (5, 0) to [out=90, in=180] (21, 4.5);	
\draw [dashed, line width=0.2mm, postaction={decorate}] (13, 0) to (21, 0);	
\draw [dashed, line width=0.2mm, postaction={decorate}] (13, 0) to [out=120, in=-90] (11, 2) to [out=90, in=180] (13, 3) to (21, 3);	
\draw [dashed, line width=0.2mm, postaction={decorate}] (3.3, 15) to [out=-90, in=180] (4, 14) to [out=0, in=-90] (4.7, 15);
\draw [dashed, line width=0.2mm, postaction={decorate}] (0.5, 7) to [out=0, in=180] (4, 8) to [out=0, in=180] (7.5, 7);
\draw [dashed, line width=0.2mm, postaction={decorate}] (0.5, 5) to [out=0, in=180] (4, 4) to [out=0, in=180] (7.5, 5);
\fill [white] (0, 12) circle(1);
\fill [white] (13, 12) circle(1);
\fill [white] (0, 0) circle(1);
\fill [white] (13, 0) circle(1);
\fill [white] (0, -12) circle(1);
\fill [white] (13, -12) circle(1);
\draw [decorate, decoration=zigzag] (3, 12) to [out=60, in=120] (5, 12);%twist front
\draw (3, 12.2) to [out=60, in=120] (5, 12.2);
\draw [decorate, decoration=zigzag] (3, 0) to [out=60, in=120] (5, 0);%twist front
\draw (3, 0.2) to [out=60, in=120] (5, 0.2);
\draw [decorate, decoration=zigzag] (6, -12) to [out=45, in=135] (8, -12);%twist front
\draw (6,0.2-12) to [out=45, in=135] (8, 0.2-12);
\node at (5.5, 12-0.8) {$p_1$};
\node at (7, 12-0.5) {$p_2$};
\node at (8.5, 12-0.5) {$p_3$};
\node at (10, 12-0.5) {$p_4$};
\node at (5.5, -0.8) {$q_1$};
\node at (7, -0.5) {$q_2$};
\node at (8.5, -0.5) {$q_3$};
\node at (10, -0.5) {$q_4$};
\node at (8.5, -0.8-12) {$r_1$};
\node at (10.5, -0.5-12) {$r_2$};
\node at (3, -0.5-12) {$r_3$};
\node at (4.5, -0.5-12) {$r_4$};
\end{tikzpicture}
\caption{}
\label{T3}
\end{figure}

Figure \ref{T3} is a framed Heegaard diagram for 3-torus $T^3$. The horizontal green, violet and orange lines are lower curve $\eta_1$, $\eta_2$ and $\eta_3$ respectively. The blue, black and red arcs give upper curve $\mu_1$, $\mu_2$ and $\mu_3$ respectively. All the Heegaard circles do not intersect with the twist fronts except at the base points, and the total rotations are: $\theta_{\eta_i}=\phi_{\eta_i}=\frac{1}{2}$ and $\theta_{\mu_j}=-\phi_{\mu_j}=-\frac{1}{2}$ for $i, j=1, 2, 3$. The rotation and the power of antipode $S$ are recorded in the following tables.

\[
\begin{tabular}{|c|c|c|c|c|c|c|c|c|c|c|c|c|} 
\hline 
& ~$p_1$~  & ~$p_2$~ & ~$p_3$~ & ~$p_4$~ & ~$q_1$~  & ~$q_2$~ & ~$q_3$~ & ~$q_4$~ & ~$r_1$~  & ~$r_2$~ & ~$r_3$~ & ~$r_4$~ \tabularnewline
\hline 
$\theta_\eta$ & ~$\frac{1}{4}$~  & ~$\frac{1}{2}$~ & ~$\frac{1}{2}$~ & ~$\frac{1}{2}$~  & ~$\frac{1}{4}$~ & ~$\frac{1}{2}$~ & ~$\frac{1}{2}$~ & ~$\frac{1}{2}$~ & ~$\frac{1}{4}$~  & ~$\frac{1}{2}$~ & ~$\frac{1}{2}$~ & ~$\frac{1}{2}$~  \tabularnewline
\hline 
$\theta_\mu$ & ~$0$~ & ~$-\tfrac{1}{4}$~ & ~$-\tfrac{1}{4}$~ & ~$\tfrac{1}{4}$~ & ~$0$~ & ~$\tfrac{1}{4}$~ & ~$-\tfrac{1}{4}$~ & ~$-\tfrac{1}{4}$~ & ~$0$~  & ~$-\tfrac{3}{4}$~ & ~$-\tfrac{1}{4}$~ & ~$-\tfrac{1}{4}$~ \tabularnewline
\hline
& ~$S$~ & ~$S^2$~ & ~$S^2$~ & ~$S$~ & ~$S$~ & ~$S$~ & ~$S^2$~ & ~$S^2$~ & ~$S$~  & ~$S^3$~ & ~$S^2$~ & ~$S^2$~ \tabularnewline
\hline
\end{tabular}
\]

The order of intersection points on $\mu_1$ is: $p_1$, $r_4$, $p_3$ and $r_2$. The order on $\mu_2$ is: $q_1$, $p_2$, $q_3$ and $p_4$. And the order on $\mu_3$ is: $r_1$, $q_2$, $r_3$ and $q_4$. The Kuperberg invariant of the 3-torus $T^3$ with this framing $f_1$ is 
\[\begin{aligned}
Z(T^3, f_1, H)=& \lambda S^{-1}(S(\Lambda^1_{(1)})S^2(\Lambda^3_{(4)})S^2(\Lambda^1_{(3)})S^3(\Lambda^3_{(2)}))\lambda S^{-1}(S(\Lambda^2_{(1)})S^2(\Lambda^1_{(2)})S^2(\Lambda^2_{(3)})S(\Lambda^1_{(4)}))\\
&\lambda S^{-1}(S(\Lambda^3_{(1)})S(\Lambda^2_{(2)})S^2(\Lambda^3_{(3)})S^2(\Lambda^2_{(4)}))\\
=& \lambda(S^2(\Lambda^3_{(2)})S(\Lambda^1_{(3)})S(\Lambda^3_{(4)})\Lambda^1_{(1)})\lambda(\Lambda^1_{(4)}S(\Lambda^2_{(3)})S(\Lambda^1_{(2)})\Lambda^2_{(1)})\\
&\lambda(S(\Lambda^2_{(4)})S(\Lambda^3_{(3)})\Lambda^2_{(2)}\Lambda^3_{(1)})
\end{aligned}\]
This invariant is the trace of certain linear map on $H^{\otimes3}$
\[\begin{aligned}
Z(T^3, f_1, H)=&\text{Tr}((S\otimes S \otimes S)\circ P_3) \\
=& \lambda^1S(\Lambda^3_{(1)}S^{-1}(\Lambda^1_{(2)})S^{-1}(\Lambda^3_{(3)})\Lambda^1_{(4)})\lambda^2S(S^{-2}(\Lambda^1_{(3)})S^{-1}(\Lambda^2_{(2)})S^{-1}(\Lambda^1_{(1)})\Lambda^2_{(4)})\\
&\lambda^3S(S^{-1}(\Lambda^2_{(3)})S^{-1}(\Lambda^3_{(2)})S^{-2}(\Lambda^2_{(1)})\Lambda^3_{(4)}),
\end{aligned}\]
where $P_3\in \mathrm{H^{\otimes3}}$ is
\[\begin{aligned}
P_3(x\otimes y\otimes z)=&z_{(1)}S^{-1}(x_{(2)})S^{-1}(z_{(3)})\otimes S^{-2}(x_{(3)})S^{-1}(y_{(2)})S^{-1}(x_{(1)})\\
&\otimes S^{-1}(y_{(3)})S^{-1}(z_{(2)})S^{-2}(y_{(1)})
\end{aligned}\]

\begin{theorem}\label{torusthm}
	Let $H$ be a finite-dimensional Hopf algebra. For the framing of 3-torus shown in Figure \ref{T3}, the Kuperberg invariant
	\[\begin{aligned}
		Z(T^3, f_1, H)=& \lambda(S^2(\Lambda^3_{(2)})S(\Lambda^1_{(3)})S(\Lambda^3_{(4)})\Lambda^1_{(1)})\lambda(\Lambda^1_{(4)}S(\Lambda^2_{(3)})S(\Lambda^1_{(2)})\Lambda^2_{(1)})\lambda(S(\Lambda^2_{(4)})S(\Lambda^3_{(3)})\Lambda^2_{(2)}\Lambda^3_{(1)})
	\end{aligned}\]
	is a gauge invariant of $H$.
\end{theorem}

\begin{proof}
By Proposition \ref{cocycle}(10), the Kuperberg invariant based on the twist Hopf algebra $H_F$ is given by 
\[\begin{aligned}
Z_F:=&Z(T^3, f_1, H_F)=\text{Tr}((S_F\otimes S_F \otimes S_F)\circ P_{3F}) \\
=& \lambda S(d^{[1]}_jf^{[1]}_r\Lambda^3_{(1)}d^{[2]}_tS_F^{-1}(f^{[2]}_i\Lambda^1_{(2)}d^{[3]}_j)S_F^{-1}(f^{[3]}_r\Lambda^3_{(3)}d^{[4]}_t)f^{[4]}_i\Lambda^1_{(4)})\\
&\lambda S(d^{[1]}_qS_F^{-2}(f^{[3]}_i\Lambda^1_{(3)}d^{[4]}_j)S_F^{-1}(f^{[2]}_p\Lambda^2_{(2)}d^{[3]}_q)S_F^{-1}(f^{[1]}_i\Lambda^1_{(1)}d^{[2]}_j)f^{[2]}_p\Lambda^2_{(4)})\\
&\lambda S(d^{[1]}_tS_F^{-1}(f^{[3]}_p\Lambda^2_{(3)}d^{[4]}_q)S_F^{-1}(f^{[2]}_r\Lambda^3_{(2)}d^{[3]}_t)S_F^{-2}(f^{[1]}_p\Lambda^2_{(1)}d^{[2]}_q)f^{[4]}_r\Lambda^3_{(4)}).
\end{aligned}\]

By Proposition \ref{cocycle}(3) and (4), we have that
\[\begin{aligned}
	Z_F
	=&\lambda S(uS(d^{[3]}_t)\Lambda^3_{(1)}S(f^{[3]}_r)u^{-1}f^{[2]}_iS^{-1}(\Lambda^1_{(2)})d^{[2]}_jf^{[1]}_rS^{-1}(\Lambda^3_{(3)})d^{[1]}_t\Lambda^1_{(4)})\\
	&\lambda S(S_F^{-1}(f^{[1]}_iS^{-1}(\Lambda^1_{(3)})d^{[1]}_j)f^{[2]}_pS^{-1}(\Lambda^2_{(2)})d^{[2]}_qf^{[3]}_iS^{-1}(\Lambda^1_{(1)})d^{[3]}_j\Lambda^2_{(4)})\\
	&\lambda S(f^{[1]}_pS^{-1}(\Lambda^2_{(3)})d^{[1]}_qf^{[2]}_rS^{-1}(\Lambda^3_{(2)})d^{[2]}_tS_F^{-1}(f^{[3]}_pS^{-1}(\Lambda^2_{(1)})d^{[3]}_q)\Lambda^3_{(4)}).
\end{aligned}\]

Apply Theorem \ref{integral}(1) for $d^{[1]}_t\Lambda^1_{(4)}$ and $d^{[3]}_j\Lambda^2_{(4)}$.

\[\begin{aligned}
Z_F
=&\lambda S(uS(d^{[3]}_t)\Lambda^3_{(1)}S(f^{[3]}_r)u^{-1}f^{[2]}_iS^{-1}(\Lambda^1_{(2)})d^{[1]}_{t(2)}d^{[2]}_jf^{[1]}_rS^{-1}(\Lambda^3_{(3)})\Lambda^1_{(4)})\\
&\lambda S(S_F^{-1}(f^{[1]}_iS^{-1}(\Lambda^1_{(3)})d^{[1]}_{t(1)}d^{[1]}_j)f^{[2]}_pS^{-1}(\Lambda^2_{(2)})d^{[2]}_qf^{[3]}_iS^{-1}(\Lambda^1_{(1)})d^{[1]}_{t(3)}d^{[3]}_j\Lambda^2_{(4)})\\
&\lambda S(f^{[1]}_pS^{-1}(\Lambda^2_{(3)})d^{[1]}_qf^{[2]}_rS^{-1}(\Lambda^3_{(2)})d^{[2]}_tS_F^{-1}(f^{[3]}_pS^{-1}(\Lambda^2_{(1)})d^{[3]}_q)\Lambda^3_{(4)})\\
=&\lambda S(uS(d^{[5]}_t)\Lambda^3_{(1)}S(f^{[3]}_r)u^{-1}f^{[2]}_iS^{-1}(\Lambda^1_{(2)})d^{[2]}_tf^{[1]}_rS^{-1}(\Lambda^3_{(3)})\Lambda^1_{(4)})\\
&\lambda S(S_F^{-1}(f^{[1]}_iS^{-1}(\Lambda^1_{(3)})d^{[1]}_t)f^{[2]}_pS^{-1}(\Lambda^2_{(2)})d^{[2]}_qf^{[3]}_iS^{-1}(\Lambda^1_{(1)})d^{[3]}_t\Lambda^2_{(4)})\\
&\lambda S(f^{[1]}_pS^{-1}(\Lambda^2_{(3)})d^{[1]}_qf^{[2]}_rS^{-1}(\Lambda^3_{(2)})d^{[4]}_tS_F^{-1}(f^{[3]}_pS^{-1}(\Lambda^2_{(1)})d^{[3]}_q)\Lambda^3_{(4)})\\
=&\lambda S(uS(d^{[5]}_t)\Lambda^3_{(1)}S(f^{[3]}_r)u^{-1}f^{[2]}_iS^{-1}(\Lambda^1_{(2)})d^{[2]}_tf^{[1]}_rS^{-1}(\Lambda^3_{(3)})\Lambda^1_{(4)})\\
&\lambda S(S_F^{-1}(f^{[1]}_iS^{-1}(\Lambda^1_{(3)})d^{[1]}_t)f^{[2]}_pS^{-1}(\Lambda^2_{(2)})d^{[3]}_{t(2)}d^{[2]}_qf^{[3]}_iS^{-1}(\Lambda^1_{(1)})\Lambda^2_{(4)})\\
&\lambda S(f^{[1]}_pS^{-1}(\Lambda^2_{(3)})d^{[3]}_{t(1)}d^{[1]}_qf^{[2]}_rS^{-1}(\Lambda^3_{(2)})d^{[4]}_tS_F^{-1}(f^{[3]}_pS^{-1}(\Lambda^2_{(1)})d^{[3]}_{t(3)}d^{[3]}_q)\Lambda^3_{(4)})\\
=&\lambda S(uS(d^{[7]}_t)\Lambda^3_{(1)}S(f^{[3]}_r)u^{-1}f^{[2]}_iS^{-1}(\Lambda^1_{(2)})d^{[2]}_tf^{[1]}_rS^{-1}(\Lambda^3_{(3)})\Lambda^1_{(4)})\\
&\lambda S(S_F^{-1}(f^{[1]}_iS^{-1}(\Lambda^1_{(3)})d^{[1]}_t)f^{[2]}_pS^{-1}(\Lambda^2_{(2)})d^{[4]}_tf^{[3]}_iS^{-1}(\Lambda^1_{(1)})\Lambda^2_{(4)})\\
&\lambda S(f^{[1]}_pS^{-1}(\Lambda^2_{(3)})d^{[3]}_tf^{[2]}_rS^{-1}(\Lambda^3_{(2)})d^{[6]}_tS_F^{-1}(f^{[3]}_pS^{-1}(\Lambda^2_{(1)})d^{[5]}_t)\Lambda^3_{(4)})
\end{aligned}\]
Plug in $S_F^{-1}(x)=S^{-1}(u)S^{-1}(x)S^{-1}(u^{-1})$.
\[\begin{aligned}
Z_F=&\lambda S(uS(d^{[7]}_t)\Lambda^3_{(1)}S(f^{[3]}_r)u^{-1}f^{[2]}_iS^{-1}(\Lambda^1_{(2)})d^{[2]}_tf^{[1]}_rS^{-1}(\Lambda^3_{(3)})\Lambda^1_{(4)})\\
&\lambda S(S^{-1}(u)S^{-1}(d^{[1]}_t)S^{-2}(\Lambda^1_{(3)})S^{-1}(f^{[1]}_i)S^{-1}(u^{-1})f^{[2]}_pS^{-1}(\Lambda^2_{(2)})d^{[4]}_tf^{[3]}_iS^{-1}(\Lambda^1_{(1)})\Lambda^2_{(4)})\\
&\lambda S(f^{[1]}_pS^{-1}(\Lambda^2_{(3)})d^{[3]}_tf^{[2]}_rS^{-1}(\Lambda^3_{(2)})d^{[6]}_tS^{-1}(u)S^{-1}(d^{[5]}_t)S^{-2}(\Lambda^2_{(1)})S^{-1}(f^{[3]}_p)S^{-1}(u^{-1})\Lambda^3_{(4)})\\
=&\lambda S(uS(d^{[5]}_t)\Lambda^3_{(1)}S(f^{[3]}_r)u^{-1}f^{[2]}_iS^{-1}(\Lambda^1_{(2)})d^{[2]}_tf^{[1]}_rS^{-1}(\Lambda^3_{(3)})\Lambda^1_{(4)})\\
&\lambda S(S^{-1}(u)S^{-1}(d^{[1]}_t)S^{-2}(\Lambda^1_{(3)})S^{-1}(f^{[1]}_i)S^{-1}(u^{-1})f^{[2]}_pS^{-1}(\Lambda^2_{(2)})d^{[4]}_tf^{[3]}_iS^{-1}(\Lambda^1_{(1)})\Lambda^2_{(4)})\\
&\lambda S(f^{[1]}_pS^{-1}(\Lambda^2_{(3)})d^{[3]}_tf^{[2]}_rS^{-1}(\Lambda^3_{(2)})S^{-2}(\Lambda^2_{(1)})S^{-1}(f^{[3]}_p)S^{-1}(u^{-1})\Lambda^3_{(4)})
\end{aligned}\]

In the last equality, $d^{[5]}_tuS(d^{[6]}_t)$ is compressed by Proposition \ref{cocycle}(6). Then we apply Theorem \ref{integral}(3) for $\lambda S(u\cdots\Lambda^1_{(4)})$ and get the following by Proposition \ref{cocycle}(7).

\[\begin{aligned}
Z_F=&\lambda S(S(d^{[5]}_t)\Lambda^3_{(1)}S(f^{[3]}_r)u^{-1}f^{[2]}_iu_{(2)}S^{-1}(\Lambda^1_{(2)})d^{[2]}_tf^{[1]}_rS^{-1}(\Lambda^3_{(3)})\Lambda^1_{(4)})\\
&\lambda S(S^{-1}(u)S^{-1}(d^{[1]}_t)S^{-2}(\Lambda^1_{(3)})S^{-1}(u_{(1)})S^{-1}(f^{[1]}_i)S^{-1}(u^{-1})f^{[2]}_pS^{-1}(\Lambda^2_{(2)})d^{[4]}_tf^{[3]}_iu_{(3)}\\
&S^{-1}(\Lambda^1_{(1)})\Lambda^2_{(4)})\\
&\lambda S(f^{[1]}_pS^{-1}(\Lambda^2_{(3)})d^{[3]}_tf^{[2]}_rS^{-1}(\Lambda^3_{(2)})S^{-2}(\Lambda^2_{(1)})S^{-1}(f^{[3]}_p)S^{-1}(u^{-1})\Lambda^3_{(4)})\\
=&\lambda S(S(d^{[5]}_t)\Lambda^3_{(1)}S(f^{[3]}_r)S(d^{[2]}_i)S^{-1}(\Lambda^1_{(2)})d^{[2]}_tf^{[1]}_rS^{-1}(\Lambda^3_{(3)})\Lambda^1_{(4)})\\
&\lambda S(S^{-1}(u)S^{-1}(d^{[1]}_t)S^{-2}(\Lambda^1_{(3)})d^{[3]}_if^{[2]}_pS^{-1}(\Lambda^2_{(2)})d^{[4]}_tuS(d^{[1]}_i)S^{-1}(\Lambda^1_{(1)})\Lambda^2_{(4)})\\
&\lambda S(f^{[1]}_pS^{-1}(\Lambda^2_{(3)})d^{[3]}_tf^{[2]}_rS^{-1}(\Lambda^3_{(2)})S^{-2}(\Lambda^2_{(1)})S^{-1}(f^{[3]}_p)S^{-1}(u^{-1})\Lambda^3_{(4)})
\end{aligned}\]

Apply Theorem \ref{integral}(3) for $\lambda S(S(d^{[5]}_t)\cdots\Lambda^1_{(4)})$ and $\lambda S(f^{[1]}_p\cdots\Lambda^3_{(4)})$. 

\[\begin{aligned}
Z_F=&\lambda S(\Lambda^3_{(1)}S(f^{[3]}_r)S(d^{[2]}_i)S(d^{[5]}_{t(2)})S^{-1}(\Lambda^1_{(2)})d^{[2]}_tf^{[1]}_rS^{-1}(\Lambda^3_{(3)})\Lambda^1_{(4)})\\
&\lambda S(S^{-1}(u)S^{-1}(d^{[1]}_t)S^{-2}(\Lambda^1_{(3)})d^{[5]}_{t(3)}d^{[3]}_if^{[2]}_pS^{-1}(\Lambda^2_{(2)})d^{[4]}_tuS(d^{[1]}_i)S(d^{[5]}_{t(1)})S^{-1}(\Lambda^1_{(1)})\Lambda^2_{(4)})\\
&\lambda S(f^{[1]}_pS^{-1}(\Lambda^2_{(3)})d^{[3]}_tf^{[2]}_rS^{-1}(\Lambda^3_{(2)})S^{-2}(\Lambda^2_{(1)})S^{-1}(f^{[3]}_p)S^{-1}(u^{-1})\Lambda^3_{(4)})\\
=&\lambda S(\Lambda^3_{(1)}S(f^{[3]}_r)S(d^{[6]}_t)S^{-1}(\Lambda^1_{(2)})d^{[2]}_tf^{[1]}_rS^{-1}(\Lambda^3_{(3)})\Lambda^1_{(4)})\\
&\lambda S(S^{-1}(u)S^{-1}(d^{[1]}_t)S^{-2}(\Lambda^1_{(3)})d^{[7]}_tf^{[2]}_pS^{-1}(\Lambda^2_{(2)})d^{[4]}_tuS(d^{[5]}_t)S^{-1}(\Lambda^1_{(1)})\Lambda^2_{(4)})\\
&\lambda S(f^{[1]}_pS^{-1}(\Lambda^2_{(3)})d^{[3]}_tf^{[2]}_rS^{-1}(\Lambda^3_{(2)})S^{-2}(\Lambda^2_{(1)})S^{-1}(f^{[3]}_p)S^{-1}(u^{-1})\Lambda^3_{(4)})\\
=&\lambda S(\Lambda^3_{(1)}S(f^{[3]}_r)S(d^{[4]}_t)S^{-1}(\Lambda^1_{(2)})d^{[2]}_tf^{[1]}_rS^{-1}(\Lambda^3_{(3)})\Lambda^1_{(4)})\\
&\lambda S(S^{-1}(u)S^{-1}(d^{[1]}_t)S^{-2}(\Lambda^1_{(3)})d^{[5]}_tf^{[2]}_pS^{-1}(\Lambda^2_{(2)})S^{-1}(\Lambda^1_{(1)})\Lambda^2_{(4)})\\
&\lambda S(f^{[1]}_pS^{-1}(\Lambda^2_{(3)})d^{[3]}_tf^{[2]}_rS^{-1}(\Lambda^3_{(2)})S^{-2}(\Lambda^2_{(1)})S^{-1}(f^{[3]}_p)S^{-1}(u^{-1})\Lambda^3_{(4)})\\
=&\lambda S(\Lambda^3_{(1)}S(f^{[1]}_{p(3)})S(f^{[3]}_r)S(d^{[4]}_t)S^{-1}(\Lambda^1_{(2)})d^{[2]}_tf^{[1]}_rf^{[1]}_{p(1)}S^{-1}(\Lambda^3_{(3)})\Lambda^1_{(4)})\\
&\lambda S(S^{-1}(u)S^{-1}(d^{[1]}_t)S^{-2}(\Lambda^1_{(3)})d^{[5]}_tf^{[2]}_pS^{-1}(\Lambda^2_{(2)})S^{-1}(\Lambda^1_{(1)})\Lambda^2_{(4)})\\
&\lambda S(S^{-1}(\Lambda^2_{(3)})d^{[3]}_tf^{[2]}_rf^{[1]}_{p(2)}S^{-1}(\Lambda^3_{(2)})S^{-2}(\Lambda^2_{(1)})S^{-1}(f^{[3]}_p)S^{-1}(u^{-1})\Lambda^3_{(4)})\\
=&\lambda S(\Lambda^3_{(1)}S(f^{[3]}_p)S(d^{[4]}_t)S^{-1}(\Lambda^1_{(2)})d^{[2]}_tf^{[1]}_pS^{-1}(\Lambda^3_{(3)})\Lambda^1_{(4)})\\
&\lambda S(S^{-1}(u)S^{-1}(d^{[1]}_t)S^{-2}(\Lambda^1_{(3)})d^{[5]}_tf^{[4]}_pS^{-1}(\Lambda^2_{(2)})S^{-1}(\Lambda^1_{(1)})\Lambda^2_{(4)})\\
&\lambda S(S^{-1}(\Lambda^2_{(3)})d^{[3]}_tf^{[2]}_pS^{-1}(\Lambda^3_{(2)})S^{-2}(\Lambda^2_{(1)})S^{-1}(f^{[5]}_p)S^{-1}(u^{-1})\Lambda^3_{(4)})
\end{aligned}\]

Plug $u=f^{[1]}_jS(f^{[2]}_j)$ into $\lambda S(S^{-1}(u)\cdots)$ and apply Theorem \ref{integral}(3) for $\lambda S(f^{[2]}_j\cdots\Lambda^2_{(4)})$.

\[\begin{aligned}
Z_F=&\lambda S(\Lambda^3_{(1)}S(f^{[3]}_p)S(d^{[4]}_t)S^{-1}(\Lambda^1_{(2)})d^{[2]}_tf^{[1]}_pS^{-1}(\Lambda^3_{(3)})\Lambda^1_{(4)})\\
&\lambda S(S^{-1}(f^{[1]}_j)S^{-1}(d^{[1]}_t)S^{-2}(\Lambda^1_{(3)})d^{[5]}_tf^{[4]}_pf^{[2]}_{j(2)}S^{-1}(\Lambda^2_{(2)})S^{-1}(\Lambda^1_{(1)})\Lambda^2_{(4)})\\
&\lambda S(f^{[2]}_{j(1)}S^{-1}(\Lambda^2_{(3)})d^{[3]}_tf^{[2]}_pS^{-1}(\Lambda^3_{(2)})S^{-2}(\Lambda^2_{(1)})S^{-1}(f^{[2]}_{j(3)})S^{-1}(f^{[5]}_p)S^{-1}(u^{-1})\Lambda^3_{(4)})\\
=&\lambda S(\Lambda^3_{(1)}S(f^{[2]}_{j(3)})S(f^{[3]}_p)S(d^{[4]}_t)S^{-1}(\Lambda^1_{(2)})d^{[2]}_tf^{[1]}_pf^{[2]}_{j(1)}S^{-1}(\Lambda^3_{(3)})\Lambda^1_{(4)})\\
&\lambda S(S^{-1}(f^{[1]}_j)S^{-1}(d^{[1]}_t)S^{-2}(\Lambda^1_{(3)})d^{[5]}_tf^{[4]}_pf^{[2]}_{j(4)}S^{-1}(\Lambda^2_{(2)})S^{-1}(\Lambda^1_{(1)})\Lambda^2_{(4)})\\
&\lambda S(S^{-1}(\Lambda^2_{(3)})d^{[3]}_tf^{[2]}_pf^{[2]}_{j(2)}S^{-1}(\Lambda^3_{(2)})S^{-2}(\Lambda^2_{(1)})S^{-1}(f^{[2]}_{j(5)})S^{-1}(f^{[5]}_p)S^{-1}(u^{-1})\Lambda^3_{(4)})\\
=&\lambda S(\Lambda^3_{(1)}S(f^{[4]}_j)S(d^{[4]}_t)S^{-1}(\Lambda^1_{(2)})d^{[2]}_tf^{[2]}_jS^{-1}(\Lambda^3_{(3)})\Lambda^1_{(4)})\\
&\lambda S(S^{-1}(f^{[1]}_j)S^{-1}(d^{[1]}_t)S^{-2}(\Lambda^1_{(3)})d^{[5]}_tf^{[5]}_jS^{-1}(\Lambda^2_{(2)})S^{-1}(\Lambda^1_{(1)})\Lambda^2_{(4)})\\
&\lambda S(S^{-1}(\Lambda^2_{(3)})d^{[3]}_tf^{[3]}_jS^{-1}(\Lambda^3_{(2)})S^{-2}(\Lambda^2_{(1)})S^{-1}(f^{[6]}_j)S^{-1}(u^{-1})\Lambda^3_{(4)})\\
\end{aligned}\]

Similarly, plug $u^{-1}=S(d^{[1]}_q)d^{[2]}_q$ in $\lambda S(\cdots S^{-1}(u)\Lambda^3_{(4)})$ and apply Theorem \ref{integral}(1) to $\lambda d^{[1]}_q\Lambda^3_{(4)}$.

\[\begin{aligned}
Z_F=&\lambda S(S(d^{[1]}_{q(3)})\Lambda^3_{(1)}S(f^{[4]}_j)S(d^{[4]}_t)S^{-1}(\Lambda^1_{(2)})d^{[2]}_tf^{[2]}_jS^{-1}(\Lambda^3_{(3)})d^{[1]}_{q(1)}\Lambda^1_{(4)})\\
&\lambda S(S^{-1}(f^{[1]}_j)S^{-1}(d^{[1]}_t)S^{-2}(\Lambda^1_{(3)})d^{[5]}_tf^{[5]}_jS^{-1}(\Lambda^2_{(2)})S^{-1}(\Lambda^1_{(1)})\Lambda^2_{(4)})\\
&\lambda S(S^{-1}(\Lambda^2_{(3)})d^{[3]}_tf^{[3]}_jS^{-1}(\Lambda^3_{(2)})d^{[1]}_{q(2)}S^{-2}(\Lambda^2_{(1)})S^{-1}(f^{[6]}_j)S^{-1}(d^{[2]}_q)\Lambda^3_{(4)})\\
=&\lambda S(S(d^{[1]}_{q(5)})\Lambda^3_{(1)}S(f^{[4]}_j)S(d^{[4]}_t)S^{-1}(\Lambda^1_{(2)})d^{[1]}_{q(2)}d^{[2]}_tf^{[2]}_jS^{-1}(\Lambda^3_{(3)})\Lambda^1_{(4)})\\
&\lambda S(S^{-1}(f^{[1]}_j)S^{-1}(d^{[1]}_t)S^{-1}(d^{[1]}_{q(1)})S^{-2}(\Lambda^1_{(3)})d^{[5]}_tf^{[5]}_jS^{-1}(\Lambda^2_{(2)})S^{-1}(\Lambda^1_{(1)})d^{[1]}_{q(3)}\Lambda^2_{(4)})\\
&\lambda S(S^{-1}(\Lambda^2_{(3)})d^{[3]}_tf^{[3]}_jS^{-1}(\Lambda^3_{(2)})d^{[1]}_{q(4)}S^{-2}(\Lambda^2_{(1)})S^{-1}(f^{[6]}_j)S^{-1}(d^{[2]}_q)\Lambda^3_{(4)})\\
=&\lambda S(S(d^{[1]}_{q(7)})\Lambda^3_{(1)}S(f^{[4]}_j)S(d^{[4]}_t)S^{-1}(\Lambda^1_{(2)})d^{[1]}_{q(2)}d^{[2]}_tf^{[2]}_jS^{-1}(\Lambda^3_{(3)})\Lambda^1_{(4)})\\
&\lambda S(S^{-1}(f^{[1]}_j)S^{-1}(d^{[1]}_t)S^{-1}(d^{[1]}_{q(1)})S^{-2}(\Lambda^1_{(3)})d^{[5]}_tf^{[5]}_jS^{-1}(\Lambda^2_{(2)})d^{[1]}_{q(4)}S^{-1}(\Lambda^1_{(1)})\Lambda^2_{(4)})\\
&\lambda S(S^{-1}(\Lambda^2_{(3)})d^{[1]}_{q(3)}d^{[3]}_tf^{[3]}_jS^{-1}(\Lambda^3_{(2)})d^{[1]}_{q(6)}S^{-1}(d^{[1]}_{q(5)})S^{-2}(\Lambda^2_{(1)})S^{-1}(f^{[6]}_j)S^{-1}(d^{[2]}_q)\Lambda^3_{(4)})\\
=&\lambda S(S(d^{[1]}_{q(5)})\Lambda^3_{(1)}S(f^{[4]}_j)S(d^{[4]}_t)S^{-1}(\Lambda^1_{(2)})d^{[1]}_{q(2)}d^{[2]}_tf^{[2]}_jS^{-1}(\Lambda^3_{(3)})\Lambda^1_{(4)})\\
&\lambda S(S^{-1}(f^{[1]}_j)S^{-1}(d^{[1]}_t)S^{-1}(d^{[1]}_{q(1)})S^{-2}(\Lambda^1_{(3)})d^{[5]}_tf^{[5]}_jS^{-1}(\Lambda^2_{(2)})d^{[1]}_{q(4)}S^{-1}(\Lambda^1_{(1)})\Lambda^2_{(4)})\\
&\lambda S(S^{-1}(\Lambda^2_{(3)})d^{[1]}_{q(3)}d^{[3]}_tf^{[3]}_jS^{-1}(\Lambda^3_{(2)})S^{-2}(\Lambda^2_{(1)})S^{-1}(f^{[6]}_j)S^{-1}(d^{[2]}_q)\Lambda^3_{(4)})\\
=&\lambda S(\Lambda^3_{(1)}S(f^{[4]}_j)S(d^{[4]}_t)S(d^{[1]}_{q(6)})S^{-1}(\Lambda^1_{(2)})d^{[1]}_{q(2)}d^{[2]}_tf^{[2]}_jS^{-1}(\Lambda^3_{(3)})\Lambda^1_{(4)})\\
&\lambda S(S^{-1}(f^{[1]}_j)S^{-1}(d^{[1]}_t)S^{-1}(d^{[1]}_{q(1)})S^{-2}(\Lambda^1_{(3)})d^{[1]}_{q(7)}d^{[5]}_tf^{[5]}_jS^{-1}(\Lambda^2_{(2)})d^{[1]}_{q(4)}S(d^{[1]}_{q(5)})S^{-1}(\Lambda^1_{(1)})\Lambda^2_{(4)})\\
&\lambda S(S^{-1}(\Lambda^2_{(3)})d^{[1]}_{q(3)}d^{[3]}_tf^{[3]}_jS^{-1}(\Lambda^3_{(2)})S^{-2}(\Lambda^2_{(1)})S^{-1}(f^{[6]}_j)S^{-1}(d^{[2]}_q)\Lambda^3_{(4)})\\
=&\lambda S(\Lambda^3_{(1)}S(f^{[4]}_j)S(d^{[4]}_t)S(d^{[1]}_{q(4)})S^{-1}(\Lambda^1_{(2)})d^{[1]}_{q(2)}d^{[2]}_tf^{[2]}_jS^{-1}(\Lambda^3_{(3)})\Lambda^1_{(4)})\\
&\lambda S(S^{-1}(f^{[1]}_j)S^{-1}(d^{[1]}_t)S^{-1}(d^{[1]}_{q(1)})S^{-2}(\Lambda^1_{(3)})d^{[1]}_{q(5)}d^{[5]}_tf^{[5]}_jS^{-1}(\Lambda^2_{(2)})S^{-1}(\Lambda^1_{(1)})\Lambda^2_{(4)})\\
&\lambda S(S^{-1}(\Lambda^2_{(3)})d^{[1]}_{q(3)}d^{[3]}_tf^{[3]}_jS^{-1}(\Lambda^3_{(2)})S^{-2}(\Lambda^2_{(1)})S^{-1}(f^{[6]}_j)S^{-1}(d^{[2]}_q)\Lambda^3_{(4)})\\
=&\lambda S(\Lambda^3_{(1)}S(f^{[4]}_j)S(d^{[4]}_q)S^{-1}(\Lambda^1_{(2)})d^{[2]}_qf^{[2]}_jS^{-1}(\Lambda^3_{(3)})\Lambda^1_{(4)})\\
&\lambda S(S^{-1}(f^{[1]}_j)S^{-1}(d^{[1]}_q)S^{-2}(\Lambda^1_{(3)})d^{[5]}_qf^{[5]}_jS^{-1}(\Lambda^2_{(2)})S^{-1}(\Lambda^1_{(1)})\Lambda^2_{(4)})\\
&\lambda S(S^{-1}(\Lambda^2_{(3)})d^{[3]}_qf^{[3]}_jS^{-1}(\Lambda^3_{(2)})S^{-2}(\Lambda^2_{(1)})S^{-1}(f^{[6]}_j)S^{-1}(d^{[6]}_q)\Lambda^3_{(4)})\\
=&\lambda S(\Lambda^3_{(1)}S^{-1}(\Lambda^1_{(2)})S^{-1}(\Lambda^3_{(3)})\Lambda^1_{(4)})\lambda S(S^{-2}(\Lambda^1_{(3)})S^{-1}(\Lambda^2_{(2)})S^{-1}(\Lambda^1_{(1)})\Lambda^2_{(4)})\\
&\lambda S(S^{-1}(\Lambda^2_{(3)})S^{-1}(\Lambda^3_{(2)})S^{-2}(\Lambda^2_{(1)})\Lambda^3_{(4)}) = Z(T^3, f_1, H)\,,
\end{aligned}\]
and this completes the proof.
\end{proof}

\bibliographystyle{alpha}
\bibliography{MyRef}

\end{document}